\documentclass{amsart}
\usepackage[T1]{fontenc}

\usepackage{amsmath}
\usepackage{amsthm}

\usepackage{amssymb}
\usepackage{amsfonts}
\usepackage{mathrsfs}
\usepackage{mathtools}
\usepackage[utf8]{inputenc}
\usepackage{enumerate}

\usepackage{natbib}
\usepackage{paralist}
\usepackage{url}
\usepackage{color}
\usepackage{framed}
\usepackage{graphicx}
\usepackage{bm}
\usepackage{comment}
\usepackage{lipsum}  % Dummytext
\usepackage{xargs}   % Use more than one optional parameter in a new command
\usepackage{xcolor}  % Coloured text etc.
\usepackage{float}
\usepackage[all,pdf]{xy}
\usepackage[ruled,vlined]{algorithm2e}
\usepackage{graphicx}
\usepackage{enumitem}
\setlist[enumerate,1]{label=(\arabic*)}       % 一级列表：1
\setlist[enumerate,2]{label=(\alph*)}        % 二级列表：a.
\setlist[enumerate,3]{label=(\roman*)} 
\setlist[enumerate]{labelindent=\parindent}

\newtheorem{theorem}{Theorem}
\newtheorem{lemma}{Lemma}

\newtheorem{proposition}{Proposition}

\newtheorem{notation}{Notation}
\newtheorem{example}{Example}
\newtheorem{remark}{Remark}

\newtheorem{property}{Property}

\newcommand{\alg}{\operatorname{alg}}

\newcommand{\Ass}{\operatorname{Ass}}

\newcommand{\diag}{\operatorname{diag}}

\newcommand{\Frac}{\operatorname{Frac}}

\newcommand{\Gal}{\operatorname{Gal}}

\newcommand{\Id}{\operatorname{Id}}

\newcommand{\Mat}{\operatorname{Mat}}

\renewcommand{\mod}{\operatorname{mod}}

\newcommand{\rank}{\operatorname{rank}}

\renewcommand{\span}{\operatorname{span}}
\newcommand{\Spec}{\operatorname{Spec}}

\newcommand{\CC}{\mathbb{C}}

\newcommand{\FF}{\mathbb{F}}

\newcommand{\QQ}{\mathbb{Q}}

\newcommand{\UU}{\mathbb{U}}

\newcommand{\ZZ}{\mathbb{Z}}
\newcommand{\Ocal}{\mathcal{O}}

\newcommand{\Afrak}{\mathfrak{A}}
\newcommand{\Bfrak}{\mathfrak{B}}

\newcommand{\Qfrak}{\mathfrak{Q}}

\newcommand{\mfrak}{\mathfrak{m}}
\newcommand{\nfrak}{\mathfrak{n}}

\newcommand{\pfrak}{\mathfrak{p}}
\newcommand{\qfrak}{\mathfrak{q}}

\newcommand{\Cbf}{\mathbf{C}}

\newcommand{\Ibf}{\mathbf{I}}

\newcommand{\Rbf}{\mathbf{R}}

\newcommand{\abf}{\mathbf{a}}
\newcommand{\bbf}{\mathbf{b}}

\newcommand{\wbf}{\mathbf{w}}
\newcommand{\xbf}{\mathbf{x}}
\newcommand{\ybf}{\mathbf{y}}
\newcommand{\zbf}{\mathbf{z}}
\newcommand{\drm}{\mathrm{d}}

\newcommand{\<}{\langle}
\renewcommand{\>}{\rangle}

\newcommand{\normof}[1]{\left\lVert#1\right\rVert}

\newcommand{\ideal}[1]{\left\langle#1\right\rangle}
\newcommand{\inner}[2]{\left\langle#1,\,#2\right\rangle}

\title{An Algorithm for Diagonalizing Matrices of Formal Power Series  }

\author{Zihao Dai}
\address[Zihao Dai]{State Key Laboratory of Mathematical Sciences, Academy of Mathematics and Systems Science, Universityof Chinese Academy of Sciences, Beijing 100190, China}
\email{daizihao@amss.ac.cn}

\author{Hao Liang}
\address[Hao Liang]{State Key Laboratory of Mathematical Sciences, Academy of Mathematics and Systems Science, University of Chinese Academy of Sciences, Beijing 100190, China}
\email{lianghao2020@amss.ac.cn}

\author{Jingyu Lu}
\address[Jingyu Lu]{University of Waterloo, Waterloo N2L3G1, ON, Canada}
\email{jingyu.lu1@uwaterloo.ca}

\author{Lihong Zhi}
\address[Lihong Zhi]{State Key Laboratory of Mathematical Sciences, Academy of Mathematics and Systems Science, University of Chinese Academy of Sciences, Beijing 100190, China}
\email{lzhi@mmrc.iss.ac.cn}

\begin{document}

\begin{abstract}

	This paper studies the unitary diagonalization of matrices over formal power series rings. Our main result shows that a normal matrix is unitarily diagonalizable if and only if its minimal polynomial completely splits over the ring and the associated spectral projections have entries in the ring. Building on this characterization, we develop an algorithm for deciding the unitary diagonalizability of matrices over regular local rings of algebraic varieties. A central ingredient of the algorithm is a decision procedure for determining whether a polynomial splits over a formal power series ring; we establish this using techniques from prime decomposition and the relative smoothness of integral closures in ramification theory.

\end{abstract}

\maketitle

\section{Introduction}\label{sec:intro}
  
    We investigate the diagonalization problem for matrices with entries in a ring of multiparameter formal power series. Denote by 
    \[\Cbf \coloneq   \mathbb{C}[\![x_1, \dots, x_m]\!]\] 
    the ring of formal power series in $m$ variables over the complex numbers. We equip  \(\Cbf\)   with an  \emph{involution} \(*\)
    that fixes the variables and acts by a complex conjugation of the coefficients, namely $c^*  \coloneqq\bar{c}$ for $  c\in\CC, x_i^* \coloneqq x_i$, for $i=1,\dots,m. 
	$
	%\begin{eqnarray*}
	 % &&c^*  \coloneqq\bar{c},\quad\ \  \text{for} \, c\in\CC, \\
   % &&x_i^* \coloneqq x_i,\quad \text{for} \, i=1,\dots,m  
	%\end{eqnarray*}
    This involution extends naturally to the matrix algebra $\Mat_n(\Cbf)$. For a matrix
    $A = (a_{ij})_{1 \leqslant i,j \leqslant n} \in \Mat_n(C),$
    we define its adjoint by
    $A^* \coloneq   (a_{ji}^*)_{1 \leqslant i,j \leqslant n}.$
    A matrix $A\in \Mat_n(\Cbf) $  is called \emph{normal} if it commutes with its adjoint, $AA^* = A^*A.$
    A matrix $U \in \Mat_n(\Cbf)$ is \emph{unitary} if 
    $U U^* = U^* U= I_n, $
    where $I_n$ denotes the identity matrix. We denote by $\UU_n(\Cbf)$ the group of unitary matrices in $\Mat_n(\Cbf)$.
    
    This paper studies the unitary diagonalization of normal matrices over formal power series rings. A matrix $A \in \Mat_n(\Cbf)$ is \emph{unitarily diagonalizable} if there exist a unitary matrix $U$ and a diagonal matrix $D$ such that
$A = U D U^*.$
Diagonalization problems for matrices over formal power series rings arise naturally in classical matrix theory—where the spectral theorem guaranties that every normal matrix over $\mathbb{C}$ is unitarily diagonalizable—and in quantum mechanics, particularly in perturbation theory, where operators are expanded as formal power series. The classic perturbation method computes the diagonalization of a matrix $A$ under the assumption that the eigenspaces of the leading term are nondegenerate. In the univariate case ($m=1$), degenerate perturbation theory yields unitary diagonalization without this nondegeneracy assumption~\cite{griffiths_introduction_2018}.
Nevertheless, rigorous results regarding the diagonalization of multiparameter formal power series matrices remain largely open. A major obstacle is that the eigenvalues of such normal matrices cannot generally be represented as formal power series; they often require Puiseux series expansions or more general algebraic extensions.

    Our work is motivated by  Parusi\'nski and Rond’s study of multiparameter perturbation theory~\cite{MR4069237}, where a sufficient condition for diagonalizability of formal power series matrices is provided using the Abhyankar-Jung theorem~\cite{MR2928451}. Their criterion requires the discriminant of the minimal polynomial to be a normal crossing divisor -- that is, a monomial times a unit -- under which diagonalizability follows. In particular, their result implies that every normal matrix over a univariate formal power series ring is diagonalizable. However, the normal crossing condition excludes some matrices that are evidently diagonalizable, as the following elementary example illustrates.
 
    \begin{example}\label{ex1}
        The discriminant of the minimal polynomial of the  diagonal matrix 
        \[A=\left(\begin{array}{cc}
        x_1 & 0\\
        0 & x_2
        \end{array}\right)\]
        is $(x_1-x_2)^2$, which is not a normal crossing divisor.
    \end{example}
 
 % This limitation motivates the search for more general conditions for diagonalizability.

%From an algorithmic standpoint, determining whether a matrix with entries in a multiparameter formal power series ring admits a diagonalization by a finite-step procedure is notoriously difficult \cite{MR4890428}. This difficulty stems from the central role of factorization over such rings: in full generality, factorization relies on techniques such as Puiseux expansions and Hensel lifting, which typically produce only truncated approximations rather than exact results. For this reason, our computational approach focuses on simpler yet practically meaningful settings. 

   % This limitation underscores the need for more general conditions with respect to diagonalizability. 

This limitation suggests that exploring more general conditions for diagonalizability could be beneficial.

%In particular, when the entries of the matrix lie in the regular local ring of an affine algebraic variety at a closed regular point, our algorithm efficiently determines whether the matrix is unitarily diagonalizable. We expect that the ideas developed here can be extended to broader classes of formal power series, including (d)-finite series.

%In this work, we make two main contributions.
%\begin{itemize}
%\item First, we provide a complete characterization of unitary diagonalizability for matrices over an arbitrary formal power series ring. This is formalized in Theorem \ref{theorem 1}, which establishes a necessary and sufficient condition for diagonalization.

%\item Second, in the geometric–local setting where the entries of $A$ lie in the regular local ring of an affine $\mathbb{C}$-algebraic variety at a closed regular point, we develop an explicit algorithm that checks diagonalizability and, when it exists, computes a corresponding unitary diagonalization. While diagonalization is generally difficult over multiparameter formal power series rings due to factorization challenges, restricting to the regular local ring allows a practical and effective computational procedure without compromising the generality of the theoretical criterion.
%\end{itemize}

    In this work, we present two primary contributions:
    \begin{itemize}
        \item \textbf{Theoretical characterization}: We establish a necessary and sufficient condition for the unitary diagonalizability of normal matrices over arbitrary formal power series rings. 
        This result, formalized in Theorem~\ref{theorem1}, extends the classical spectral theorem to the multiparameter setting, providing a rigorous algebraic foundation applicable to matrices dependent on an arbitrary finite number of parameters.

        \item \textbf{Algorithmic construction}: We develop a constructive algorithm for  matrices with entries in the regular local ring of an affine $\mathbb{C}$-variety at a closed regular point. A key ingredient is a computable criterion (Theorem~\ref{Thm:PowerSeriesCriteria4Jacobian}) that bypasses the intractability of factorization in formal power series rings: it characterizes when the minimal polynomial splits completely in the local ring via a full-rank condition on a specific Jacobian matrix. This reduces the existence of a unitary decomposition to a simple rank test and enables its explicit computation without resorting to root-finding in power series rings.
    \end{itemize}

\section{Criterion for Unitary Diagonalization over $\Cbf$}\hfill\label{sec:criterion}

    In this section, we establish a necessary and sufficient criterion for unitary diagonalization over $\Cbf$. Let 
    $\Rbf:=\mathbb{R}[\![x_1, \dots, x_m]\!]$
    be the  ring of formal power series in the $m$ variables in the real number field, which is the subring of the involution invariant subring of $\Cbf=\mathbb{C}[\![x_1, \dots, x_m]\!]$. The power series rings $\Cbf$ and $\Rbf$ are regular local rings with maximal ideals $\mfrak=\mfrak_{\Cbf}$ and $\mfrak_{\Rbf}$ both generated by the common regular local real parameters $\xbf=(x_1, \dots, x_m)$, respectively. The groups of units in $\Cbf$ and $\Rbf$ are, respectively, $\Cbf^{\times}\coloneqq \Cbf\setminus \mfrak$ and $\Rbf^{\times}\coloneqq \Rbf\setminus \mfrak_{\Rbf}$. We equip $\Cbf^n$ with the standard Hermitian inner product: 
    \[\inner{-}{-}:\Cbf^n \times \Cbf^n \to \Cbf,\quad (\abf,\bbf) \mapsto \sum a_kb_k^*.\]
    For a vector $\abf\in \Cbf^n$, the inner product $\inner{\abf}{\abf}$ always lies in $\Rbf$. If $\inner{\abf}{\abf}$ is invertible in $\Rbf$ (i.e. invertible in $\Cbf$), we then define the $\Cbf$-norm of $\abf$ by $\normof{\abf}\coloneqq \sqrt{\inner{\abf}{\abf}} \in \Rbf$, normalized so that $\normof{\abf}(0)= \normof{\abf}\mod \mfrak > 0$.
Note that if $\inner{\abf}{\abf}$ is not invertible in $\Rbf$ (for example, if $\abf=[x_1,x_2]^{\top}$), then $\sqrt{\inner{\abf}{\abf}}$ need not be an element of $\Rbf$ or $\Cbf$.  
   % Note that if $\inner{\abf}{\abf}$ is not invertible in $\Rbf$, for instance $\abf=[x_1,x_2]^{\top}$, then $\sqrt{\inner{\abf}{\abf}}$ is neither necessarily in $\Rbf$ nor in $\Cbf$.
  
    The proof of Theorem \ref{theorem1} is inspired by the classic perturbation theory, which in essence is a Hensel lifting procedure~\cite{Hensel1904}. Hensel lifting is a powerful algorithmic tool for solving equations over a formal power series ring, yielding a degree-by-degree approximation of matrix diagonalization under the non-degeneracy assumption, thereby simultaneously establishing both existence and computability. In proving Theorem \ref{theorem1}, we refine this approach to remove the non-degeneracy hypothesis. Furthermore, our main result in Section \ref{sec:split} employs a more sophisticated variant (Henselization). Since Henselian ideas remain central throughout, we first review the basic setup.
    %We will use Hensel lifting in the proof of Theorem \ref{theorem 1}, which is a refinement of this simple idea and works without assuming non-degenerateness.  Moreover, our main theorem in Section \ref{sec:split} will take advantage of a more sophisticated technique (Henselization). One will quickly notice the relic of the principles of Hensel lifting in the proof. So before we enter the details of Theorem \ref{theorem 1}, it is worth presenting a rather brief review of this approach. 

    Consider the case $m = 1$ and let $A = A_0 + A_1 x_1$ be with $A_0, A_1 \in \operatorname{Mat}_n(\mathbb{C})$. Assume that the characteristic polynomial $\chi_{A_0}(t)$ has $n$ distinct roots (the non‑degenerate case). For $\ybf = (y_1, \dots, y_n)$, let \(e_k(\ybf)\) denote the $k$-th elementary symmetric polynomial. Define the polynomial map
    %Say $m=1$ and $A = A_0 + A_1 x_1$ and the characteristic polynomial $\chi_{A_0}(t)$ has $n$ different roots (the non-degenerate assumption), where $A_0, A_1\in \Mat_n(\CC)$. Let $e_k(y)\coloneq \sum_{1\leq j_1\leq\ldots \leq j_k \leq n}\prod_{l=1}^k y_{j_l}$ be the $k$-th elementary polynomial of $y_1,\ldots y_n$. Consider the polynomial map induced by $(e_{k})_{k=1}^{n}$:
    \begin{equation*}
        f: \Cbf^n \to \Cbf^n, \quad \ybf=(y_1,\ldots,y_n) \mapsto (-e_1(\ybf),\ldots,(-1)^ne_n(\ybf)).
    \end{equation*}
    Because $A_0$ has no multiple eigenvalues, the tangent map $\drm f$ has full rank at the origin $o\in\Cbf^n$. The Implicit Function Theorem (IFT) for the formal power series ring then ensures that $\chi_A(t)\in\Cbf[t]$ splits completely over $\Cbf$. Starting from the normalized eigenvector $\bar{v}\in\CC^n$ of $A_0$ associated with the eigenvalue $\lambda\in \CC$, the Newton-Hensel iteration lifts $\bar{v}$ degree-wise to a normalized eigenvector $v\in\Cbf^n$ of $A$, associated with $\lambda\in\Cbf$ with $\lambda\equiv\bar{\lambda}\mod \mfrak_\Cbf$.
    
    If $\chi_{A_0}(t)$ has multiple roots (degenerate cases), $\rank(\drm f|_o) < n$, then IFT fails and the Newton-Hensel iteration algorithm is ill-defined, since it requires inverting $\drm f|_o$. More refined techniques are needed to address such degeneracies~\cite{MR1907381,sasaki2000,MR3663235}. 

\begin{comment} 
{\color{red}
    Since $A_0$ has no multiple roots, the differential $\drm f$ is of full rank at the origin $o \in \Cbf^n$. The Inverse Function Theorem (IFT) for the formal power series ring $\Cbf$ then guarantees that the characteristic polynomial $\chi_A(t) \in \Cbf[t]$ splits over $\Cbf$.

Starting from a normalized eigenvector $v \in \CC^n$ of $A_0$ associated with an eigenvalue $\lambda \in \CC$, the Newton--Hensel iteration lifts $v$ degree by degree to a normalized eigenvector $v' \in \Cbf^n$ of $A$, associated with an eigenvalue $\lambda' \in \Cbf$ satisfying
\[
\lambda' \equiv \lambda \pmod{\langle x_1 \rangle}.
\]

By contrast, if $\chi_A(t)$ has multiple roots (thus violating the non-degeneracy assumption), then
\[
\rank(\drm f|_o) < n.
\]
In this case, the IFT no longer applies, and the Newton--Hensel iteration becomes ill-defined, since it requires inversion of $\drm f|_o$. Dedicated techniques are therefore required to handle the degenerate cases arising in Hensel iteration \cite{MR1907381,sasaki2000,MR3663235}.}
\end{comment}
    
    Given a normal matrix $A\in \Mat_n(\Cbf)$, suppose $A$ has $d$ distinct eigenvalues $\lambda_k\in\FF\coloneqq\overline{\Frac(\Cbf)}^{\alg}$ with multiplicity $n_k$, $k = 1,\ldots,d$ and $\sum n_k = n$.  The following theorem characterizes when $A$ is unitarily diagonalizable.

    \begin{theorem}\label{theorem1}
        A normal matrix $A$ is diagonalizable under the adjoint action of the unitary group $\UU_n(\Cbf)$ if and only if the characteristic polynomial of $A$ totally splits over $\Cbf$ and all the matrix entries of the projection operators onto the eigen-spaces of $A$ belong to $\Cbf$, i.e. there exists  $ \Pi_k\in \Mat_n(\Cbf)$, such that 
        \begin{equation}\label{eq: orth decomp}
            \Pi_{k}^{*} = \Pi_{k}, ~ \Pi_{k}^2 = \Pi_{k}, ~ \sum_k \Pi_{k} = I_n, ~ \Pi_{k}\Pi_{k'}=0 ~ \text{if} ~ k\neq k',
        \end{equation}
        and $A = \sum_k\lambda_k \Pi_{k}$.
    \end{theorem}
    \begin{proof}

        Suppose that $A$ is diagonalizable, i.e., there are a unitary matrix $U\in \UU_n(\Cbf)$ and a diagonal matrix $D\in\Mat_n(\Cbf)$ such that $A = UDU^{-1}$. The diagonal entries of $D$ are the eigenvalues of $A$, so the characteristic polynomial of $A$, $\chi_A\in \Cbf[t]$ splits over $\Cbf$. Up to a permutation of the coordinate, we write $D = \diag(\lambda_1I_{n_1},\ldots,\lambda_dI_{n_d})$, where $n_k$ is the multiplicity of the eigenvalue $\lambda_k$, and $\sum_k{n_k}=n$. Set 
        \[\Pi_{k}\coloneqq U\diag(0,\ldots,0,I_{n_k},0,\ldots,0)U^{-1}\in\Mat_n(\Cbf).\] 
        Straightforward verification shows that (\ref{eq: orth decomp}) is satisfied. Here $\Pi_{k}$ is the orthogonal projection operator onto the eigenspace $V_{k}$ associated with the eigenvalue $\lambda_k$. 

        Conversely, suppose that $A = \sum_k\lambda_k \Pi_{k}$ with $\Pi_{k}$ the projection operator onto the eigen-value $\lambda_k$'s eigen-space $V_k$ which satisfies (\ref{eq: orth decomp}), and $\Pi_{k}\Pi_{k'}=0$ when $k\neq k'$. The property of the orthogonal projection operators $\Pi_{k}\in\Mat_n(\Cbf)$ implies $\Cbf^n=\bigoplus_k (\Pi_{k}\Cbf^n)$, and each $V_k=\Pi_{k}\Cbf^n$ is a finite projective $\Cbf$-module. Since $\Cbf$ is a local ring, one deduces that every $V_k$ is a free $\Cbf$-module. We need to extract an orthonormal $\Cbf$-basis of $V_k$ from $\Pi_{k}$. Denote $\Pi_{k}(0)\in\Mat_n(\CC)$ as the remainder of $\Pi_{k} \mod \mfrak$.  Then we have 
        \begin{equation*}
            \Pi_{k}^{*}(0) = \Pi_{k}(0), \Pi_{k}(0)^2 = \Pi_{k}(0),\sum_k \Pi_{k}(0) = I_n,\Pi_{k}(0)\Pi_{k'}(0)=0
        \end{equation*} when $k\neq k'$. This is an orthogonal decomposition of $\CC^n$. We recover the Gram-Schmidt orthogonalization over $\Cbf$ from the classic one over $\CC$. Applying Gram-Schmidt orthogonalization over $\CC$, we obtain an orthonormal $\CC$-basis of the image of $\Pi_{k}(0)$ in $\CC^n$, say $e_{k,1},\ldots,e_{k,n_{k}}\in \CC^{n_k}$.  We can write $e_{k,j} = \Pi_{k}(0)c_{k,j}$ with $c_{k,j}\in \CC^n$ and $\normof{c_{k,j}} > 0$. 
        %  Since  \[\inner{\Pi_{k}c_{k,j}}{\Pi_{k}c_{k,j}}\equiv \normof{e_{k,j}}^2 =1\mod \mfrak,\]
      %   by Hensel lifting over $\Cbf$, $\inner{\Pi_{k}c_{k,j}}{\Pi_{k}c_{k,j}}$ is an invertible square in $\Cbf$, which means we can find $w_{k,j}\in \Cbf$ such that $w_{k,j}^2 = \inner{\Pi_{k}c_{k,j}}{\Pi_{k}c_{k,j}}$ and $w_{k,j}(0) = 1$. 
      
Since
\[
    \inner{\Pi_{k}c_{k,j}}{\Pi_{k}c_{k,j}} \equiv \normof{e_{k,j}}^2 = 1\  (\mod\mfrak),
\]
this inner product is a unit in $\Rbf$. By Hensel's Lemma, there exists a unique element $w_{k,j} \in \Rbf^{\times}$ such that 
\[w_{k,j}^2 = \inner{\Pi_{k}c_{k,j}}{\Pi_{k}c_{k,j}}, ~w_{k,j}(0) = 1.\]
 
 Set $E_{k,j}\coloneqq\frac{\Pi_{k}c_{k,j}}{w_{k,j}}$. $E_{k,j}\in \Cbf^n$ as $w_{k,j}\notin\mfrak$ is invertible in $\Cbf$. The definition of $E_{k,j}$ ensures
        \begin{equation}\label{eq: orthonormal mod m}
            \normof{E_{k,j}} = 1,\quad \inner{E_{k,j}}{E_{k,j'}} \equiv \delta_{j,j'}\ (\mod \mfrak),
        \end{equation} 
 where $\delta_{j,j'}$ denotes the Kronecker delta, equal to $1$ if $j=j'$ and $0$ otherwise.      
 %Equation (\ref{eq: orthonormal mod m}) implies that $\span_{\Cbf}(E_{k,j'})_{j'\leq j}$ is a free $\Cbf$-module for any $j = 1,\ldots,n_k$, since $\Cbf$ is an integral local ring.
 Equation (\ref{eq: orthonormal mod m}) implies that $\operatorname{span}_{\Cbf}(E_{k,j'})_{j'\leqslant j}$ is a free $\Cbf$-module for any $j = 1,\ldots,n_k$, utilizing the fact that $\Cbf$ is a local ring.
 
 Set $\tilde{e}_{k,1}\coloneqq E_{k,1}$. Define $\tilde{e}_{k,j}\in \span_{\Cbf}(E_{k,j'})_{j'\leqslant j}\setminus \span_{\Cbf}(E_{k,j'})_{j'< j}$ inductively as 
        \begin{equation}\label{eq: gram-schmidt}
            \tilde{e}_{k,j}\coloneqq \frac{E_{k,j}-\sum_{j'<j}\inner{E_{k,j}}{\tilde{e}_{k,j'}}\tilde{e}_{k,j'}}{\normof{E_{k,j}-\sum_{j'<j}\inner{E_{k,j}}{\tilde{e}_{k,j'}}\tilde{e}_{k,j'}}}.
        \end{equation}
        We verify that the definition of $(\tilde{e}_{k,j'})_{j'=1}^{n_k}$ is valid and gives an orthonormal $\Cbf$-basis of the free module $\Pi_{k}\Cbf^n$ by induction. $\tilde{e}_{k,1}$ is well defined. Suppose that for $j'<j$, $(\tilde{e}_{k,j'})_{j'<j}$ is well defined and gives an orthonormal $\Cbf$-basis of the free module $\span_{\Cbf}(E_{k,j'})_{j'< j}$. Then we have 
        \[E_{k,j}-\sum_{j'<j}\inner{E_{k,j}}{\tilde{e}_{k,j'}}\tilde{e}_{k,j'} \equiv e_{k,j}\ (\mod \mfrak),\]
        which implies that 
        \[\normof{E_{k,j}-\sum_{j'<j}\inner{E_{k,j}}{\tilde{e}_{k,j'}}\tilde{e}_{k,j'}}\in \Cbf^{\times}.\]
        Hence, $\tilde{e}_{k,j}$ is well defined. By induction we see that Equation (\ref{eq: gram-schmidt}) implies that $E_{k,j}$ can be $\Cbf$-linearly expressed in terms of $(\tilde{e}_{k,j'})_{j'\leqslant j}$ and vise versa,  which ensures that $(\tilde{e}_{k,j'})_{1\leqslant j'\leqslant j}$ gives an orthonormal $\Cbf$-basis of $\span_{\Cbf}(E_{k,j'})_{j'\leqslant j}$. 

%With all $(\tilde{e}_{k,j})_{1\leq k \leq d, 1\leq j\leq n_k}$, set $U\coloneqq [\tilde{e}_{1,1},\ldots,\tilde{e}_{1,n_1},\tilde{e}_{2,1},\ldots,\tilde{e}_{k,n_k}]$, then a straightforward computation verifies $AU = U\diag(\lambda_1I_{n_1},\ldots,\allowbreak \lambda_dI_{n_d})$ as desired.

Let $U$ be the matrix with columns given by the collection $(\tilde{e}_{k,j})_{k,j}$, ordered as
\[
    U \coloneqq [\tilde{e}_{1,1}, \ldots, \tilde{e}_{1,n_1}, \ldots, \tilde{e}_{d,1}, \ldots, \tilde{e}_{d,n_d}].
\]
A straightforward computation then verifies that 
\[AU = U\diag(\lambda_1I_{n_1}, \ldots, \lambda_dI_{n_d}). \]

\end{proof}
%\begin{remark}
  %  Through the injection $\Cbf \hookrightarrow \Frac(\Cbf)$, for given normal matrix $A$, we can always find projection operators $\Pi_{k}\in \Mat_n(\Frac(\Cbf))$ which satisfy (\ref{eq: orth decomp}) when $\chi_A(t)$ splits in $\Cbf$. In fact these $(\Pi_{k})_k$ are uniquely determined by Lagrange interpolation on the $d$ distinct roots of $\chi_A(t)$. When $\Pi_{k}\in \Mat_n(\Cbf)$, these $(\Pi_{k})_k$ serve as the projection operators in the statement of Theorem \ref{theorem 1}.
  %  \end{remark}
\begin{remark}

 Using the natural injection $\Cbf \hookrightarrow \Frac(\Cbf)$, we consider the normal matrix $A$ as an element of $\Mat_n(\Frac(\Cbf))$. If the characteristic polynomial $\chi_A(t)$ splits over $\Cbf$ (equivalently, over $\Frac(\Cbf)$), we can construct projection operators
$
    \Pi_{k} \in \Mat_n(\Frac(\Cbf))
$
satisfying~\eqref{eq: orth decomp} and $A=\sum_k{\lambda_k \Pi_{k}}$. These projections, called the \emph{spectral projections} of $A$, are uniquely determined and can be computed using Lagrange interpolation on the $d$ distinct roots of $\chi_A(t)$. Furthermore, if $\Pi_{k} \in \Mat_n(\Cbf)$ for all $k$, then this family $(\Pi_{k})_k$ constitutes the spectral projections required for Theorem~\ref{theorem1}. 
\end{remark}

\begin{comment}
    \begin{remark}
 By natural injection $\Cbf \hookrightarrow \Frac(\Cbf)$, for any normal matrix $A$ whose characteristic polynomial $\chi_A(t)$ splits over $\Cbf$, one can construct projection operators 
\[
\Pi_{k} \in \Mat_n(\Frac(\Cbf))
\]
satisfying~\eqref{eq: orth decomp}. These projections are uniquely determined by Lagrange interpolation on the $d$ distinct roots of $\chi_A(t)$. 
Moreover, if, in fact, $\Pi_{k} \in \Mat_n(\Cbf)$, then the family $(\Pi_{k})_k$ provides the projection operators appearing in the statement of Theorem~\ref{theorem 1}.
\end{remark}
\end{comment}

		%One can deduce the unitary diagonalization of a normal matrix \(A\in \Mat_n(\Cbf)\) by the following steps:
The unitary diagonalization of a normal matrix $A \in \Mat_n(\Cbf)$ can be computed by following the standard steps described below.  %However, over formal power series rings, this approach is algorithmically infeasible: exact spectral decomposition requires eigenvalues in the algebraic closure (often infinite Puiseux series), precluding finite symbolic computation.

		\begin{enumerate}\upshape
			\item Compute $A$'s (monic) minimal polynomial 
            $$\mu_A(t)=\chi_A(t)/\gcd(\chi_A(t),\chi_A'(t));$$
             factor it into the product of irreducible polynomials in $\Cbf[t]$: $$\mu_A(t)=p_1(t)\cdots p_d(t).$$
             If any \(p_k(t)\) is non-linear, return \textbf{NOT Diagonalizable}.
      
            \item For  $k=1,\dots,d$, compute the projection polynomial $P_k(t)\in \Cbf[t]$ by Lagrange interpolation, which satisfies $$P_k(t)\equiv \delta_{kl} ~(\mod p_l(t)).$$

             If the denominator of any reduced entry of the projection matrix $\Pi_k=P_k(A)$ does not belong to $\Cbf^\times$, return \textbf{NOT Diagonalizable}.

            \item For  $k=1,\dots,d$, find the orthonormal $\Cbf$-basis $B_k$ of $\Pi_k\Cbf^n$ (e.g., via Gram-Schmidt). Let $n_k=\rank B_k$. Return $U=(B_1|\dots|B_d)$ and $D=\diag(\lambda_1 I_{n_1},\dots,\lambda_d I_{n_d})$. 
		\end{enumerate}

In the following, we show an example that does not 
 satisfy the sufficient condition given in \cite[Theorem 2.5]{MR4069237}  but is nevertheless diagonalizable.
 \begin{example}
        Consider the unitary diagonalization of 
        \begin{equation*}
            A=\begin{pmatrix}
                \frac{1}{2}(x_1+x_2) & -\frac{i}{2}(x_1-x_2) \\
                \frac{i}{2}(x_1-x_2) & \frac{1}{2}(x_1+x_2)
            \end{pmatrix},
        \end{equation*}
        over $\Cbf=\CC[\![x_1,x_2]\!]$. 
        The characteristic polynomial of $A$ is 
        \[\chi_A(t)=t^2-(x_1+x_2) t+x_1x_2,\]
        which coincides with  the  minimal polynomial $\mu_A(t) $. Its discriminant 
    is $(x_1-x_2)^2$, which  is not a unit of $\Cbf = \mathbb{C}[\![x_1, x_2]\!]$. Therefore,  it 
 does not satisfy the sufficient condition of \cite[Theorem 2.5]{MR4069237}.
 
The projection polynomials  corresponding to $p_1(t)=t-x_1$ and $p_2(t)=t-x_2$ are  
     \[P_1(t)=\frac{t-x_2}{x_1-x_2}, \quad P_2(t)=\frac{t-x_1}{x_2-x_1}.\]
     Compute the projection matrices 
        \begin{equation*}
            \Pi_1=P_1(A)=\begin{pmatrix}
                \frac{1}{2} & -\frac{i}{2} \\
                \frac{i}{2} & \frac{1}{2}
            \end{pmatrix},\quad
            \Pi_2=P_2(A)=\begin{pmatrix}
                \frac{1}{2} & \frac{i}{2} \\
                -\frac{i}{2} & \frac{1}{2}
            \end{pmatrix}.
        \end{equation*}
The orthonormal bases of the subspaces $\Pi_1\Cbf^2$ and $\Pi_2\Cbf^2$ are given respectively by
\[
    B_1 = \frac{1}{\sqrt{2}}\begin{pmatrix} 1 \\ i \end{pmatrix}
    \quad \text{and} \quad
    B_2 = \frac{1}{\sqrt{2}}\begin{pmatrix} i \\ 1 \end{pmatrix}.
\]
Therefore, the diagonalization of $A$
    over $\Cbf=\CC[\![x_1,x_2]\!]$  is given by 
        $A=U D U^*$, where 
        \begin{equation*}
            U=\begin{pmatrix}
                \frac{1}{\sqrt{2}} & \frac{i}{\sqrt{2}} \\
                \frac{i}{\sqrt{2}} & \frac{1}{\sqrt{2}}
            \end{pmatrix},~\text{and}~
            D=\begin{pmatrix}
                x_1 & 0 \\
                0 & x_2
            \end{pmatrix}.
        \end{equation*}
\end{example}

    \begin{example}\label{exp:adam's ctexp}
        Consider the following example in \cite[Example 6.1]{MR2372149}
        \begin{equation*}
            A=\begin{pmatrix}
                x_1^2 & x_1 x_2 \\
                x_1 x_2 & x_2^2
            \end{pmatrix}.
        \end{equation*}
        The  minimal  polynomial of $A$ is $$ \mu_A(t)=t^2-(x_1^2+x_2^2)t. $$
        %which coincides with its minimal polynomial over $\Cbf=\CC[\![x_1,x_2]\!]$. 
        Factoring $\mu_A(t)=t(t-(x_1^2+x_2^2))$ over $C$, the irreducible factors are 
        $$ p_1(t)=t, \quad p_2(t)=t-(x_1^2+x_2^2).$$ 
        The projection polynomials corresponding to $p_1(t)$ and  $p_2(t)$ are 
        \[P_1(t)=\frac{t}{x_1^2+x_2^2}, \quad P_2(t)=-\frac{t-(x_1^2+x_2^2)}{x_1^2+x_2^2}.\]
        Compute the projection matrices 
        \begin{eqnarray*}
            &&\Pi_1=P_1(A)=\frac{1}{x_1^2+x_2^2}\begin{pmatrix}
                x_1^2 & x_1 x_2\\
                x_1 x_2& x_2^2
            \end{pmatrix},\\
            &&\Pi_2=P_2(A)=\frac{1}{x_1^2+x_2^2}\begin{pmatrix}
                x_2^2 & -x_1 x_2 \\
                -x_1 x_2 & x_1^2
            \end{pmatrix}.
        \end{eqnarray*}
        
        The denominator of both $\Pi_1$ and $\Pi_2$ is
        $x_1^2+x_2^2 \notin \Cbf^{\times}.$
        According to Theorem \ref{theorem1}, $A$ cannot be unitarily diagonalized over $\Cbf$.
    \end{example}

In the two examples above, the eigenvalues of the given matrices are polynomials in  $x_1, \ldots, x_m$, and hence the minimal polynomial $\mu_A(t)$  splits into linear factors. In general, however, the eigenvalues of a matrix with entries in a formal power series ring need not admit any explicit finite expression over $\Cbf$. 

%Factorization over the formal power series ring
%\[\Cbf \coloneq   \mathbb{C}[\![x_1,\dots,x_m]\!]\]
%presents computational difficulties that are qualitatively different from those arising in polynomial rings. Although $\Cbf$ is a unique factorization domain \cite[Chapter 2, lemma 2.2]{BAG2013I}, its elements are infinite series, and factorizations are typically governed by subtle analytic and algebraic structures.

From an algorithmic perspective, it is computationally intractable to determine whether a matrix over a multiparameter formal power series ring is diagonalizable via a finite procedure. %~\cite{MR4890428}.
The core obstacle lies in factorization.
Factorization over the formal power series ring
%\[\Cbf \coloneq   \mathbb{C}[\![x_1,\dots,x_m]\!]\]
presents computational difficulties that are qualitatively different from those that arise in polynomial rings \cite{MR682664,MR784750,Kaltofen1986,Kaltofen1991,MR4469154,MR4546085}. Although $\Cbf$ is a unique factorization domain \cite[Chapter 2, lemma 2.2]{BAG2013I},  % its elements are infinite series, and factorizations are typically governed by subtle analytic and algebraic structures.
 standard tools such as Puiseux expansions \cite{puiseux1850recherches} and Hensel lifting \cite{Hensel1904, ZASSENHAUS1969291,MR1907381} typically yield only truncated approximations rather than exact closed form factors~\cite{gathen2013modern, becker1993grobner}. To address this, our computational framework focuses on algebraic-geometric local rings, which admit effective finite-step algorithms, specifically via Mora's tangent cone algorithm~\cite{mora1982algorithm} and standard basis techniques~\cite{greuel2008singular} without sacrificing theoretical generality.

    \section{When are the eigenvalues power series?}\label{sec:split}
    
    Suppose now that the entries of $A$ belong to a regular local ring arising from a $\CC$-algebraic variety. More precisely, let \(X\) be an irreducible affine $\CC$-variety of dimension \(m\) embedded in \(\mathbb{C}^M\), and let the base point \(o\in\CC^M\) be a regular point of \(X\). Denote by $\xbf=(x_1,\dots,x_M)$ the coordinate variables of $\CC^M$, and by $\mathbf{I}(X)$ the defining ideal of $X$ in $\CC^M$. The affine coordinate ring of $X$ is then
    \[\Ocal_X(X)=\CC[\xbf]/\mathbf{I}(X).\] 
   % where $\Ocal_X$ is the structure sheaf of $X$. 
    
    The maximal ideal of $\Ocal_X(X)$ corresponding to $o$ is $\xbf\Ocal_X(X)$, and the regular local ring \(\Ocal_{X,o}\) is the localization of $\Ocal_X(X)$ at $\xbf\Ocal_X(X)$. Denote by $$\mfrak_{X,o}=\xbf\Ocal_{X,o}$$ the unique maximal ideal of $\Ocal_{X,o}$, and by $K=\Frac(\Ocal_{X,o})$ the fraction field of $\Ocal_{X,o}$. 
    %For any \(c\in\Ocal_{X,o}\), \(c=F/G\) where \(F,G\in \mathbb{C}[\xbf]/\mathbf{I}(X)\) and \(G(o)\neq 0\). 
    %Since $\dim(X)=m$ and $o\in X$ is regular, one deduces that $\Ocal_{X,o}$ is a regular local ring of dimension $m$. 
    
    Let $\mu(t)\in\Ocal_{X,o}[t]$ be a monic polynomial in variable $t$ over $\Ocal_{X,o}$. The purpose of this section is to determine whether all roots of $\mu$ lie in the $\mfrak_{X,o}$-adic completion $\widehat{\Ocal_{X,o}}$ of $\Ocal_{X,o}$. Here $\widehat{\Ocal_{X,o}}$ is identified with the power series ring $\Cbf$ in $m$ variables -- the regular local parameters of $\mfrak_{X,o}$ -- over $\CC$ . 
    %Applying Euclid's algorithm on $\mu$, we may assume $\mu$ is square-free if necessary. 
    
    Write 
    $$\mu(t)=t^d+c_1t^{d-1}+\cdots+c_d,$$ 
    with $c_k\in\Ocal_{X,o},\,k=1,\dots,d$. Introduce variables $\ybf=(y_1,\dots,y_d)$ to represent the roots of $\mu$. For $k=1,\dots,d$, denote by  
    $$e_k(\ybf)=\sum_{1\leqslant j_1<\dots<j_k\leqslant d}{y_{j_1}\dots y_{j_k}}.$$
    the $k$-th elementary symmetric polynomial in $\ybf$. Let $J$ be the ideal of $\Ocal_{X,o}[\ybf]$ generated by \textit{Vieta's polynomials} 
    \begin{equation}\label{Vieta}
        \xi_k(\ybf)\coloneq  e_k(\ybf)-(-1)^k c_k,\quad~\text{for}~k=1,\dots,d. 
    \end{equation}

The quotient ring $\Ocal_{X,o}[\ybf]/J$ is known as the \emph{universal decomposition algebra}~\cite{Bourbaki_Algegre_4, Lebreton&Schost_2012} or the \emph{splitting algebra}~\cite{Ekedahl&Laksov_2005, Laksov_2010, Thorup_2011, Laksov&Thorup_2012} of $\mu$ over $\Ocal_{X,o}$. This algebra encodes information regarding the roots of $\mu$ without assuming their specific form. Let $\Ass(\Ocal_{X,o}[\ybf]/J)$ denote the set of associated primes of the $\Ocal_{X,o}[\ybf]$-module $\Ocal_{X,o}[\ybf]/J$. We have the following result.

    %\begin{lemma}
    %    $e_1(\ybf),\dots,e_d(\ybf)\in\Ocal_{X,0}[\ybf]$ is algebraically independent over $\Ocal_{X,0}$, and $\Ocal_{X,0}[\ybf]$ is a finite free $\Ocal_{X,0}[\ybf]^{\Sigma_d}$-module of rank $d!$ with basis $\prod_{k=1}^d{y_k^{j_k}}$ for $0\leqslant j_k<k$. 
    %\end{lemma}
    \begin{proposition}
        $\Ass(\Ocal_{X,o}[\ybf]/J)$ are exactly the minimal primes containing $J$ in $\Ocal_{X,o}[\ybf]$, %and the height of any minimal primes containing $J$ in $\Ocal_{X,0}[\ybf]$ is $d$. 
        and are conjugate under the permutations on the variables $\ybf=(y_1,\dots,y_d)$. Moreover, the residue field $\kappa(Q)$ of $\Ocal_{X,o}[\ybf]$ at any $Q\in\Ass(\Ocal_{X,o}[\ybf]/J)$ is a splitting field of $\mu(t)$ over the fraction field $K$ of $\Ocal_{X,o}$. 
    \end{proposition}
    \begin{proof}
        It is well known that the invariant subalgebra of $\Ocal_{X,o}[\ybf]$ under the permutations $\Sigma_d$ on the variables $\ybf=(y_1,\dots,y_d)$ is 
        $$\Ocal_{X,o}[\ybf]^{\Sigma_d}=\Ocal_{X,o}[e_1(\ybf),\dots,e_d(\ybf)],$$
        the symmetric polynomial algebra over $\Ocal_{X,o}$, which itself is a polynomial algebra over $\Ocal_{X,o}$ in variables $e_1(\ybf),\dots,e_d(\ybf)$. Moreover, $\Ocal_{X,o}[\ybf]$ is a finite free $\Ocal_{X,o}[\ybf]^{\Sigma_d}$-module of rank $d!$ with a basis $y_1^{k_1}\cdots y_d^{k_d}$ for $0\leqslant k_j<j$ and $1\leqslant j\leqslant d$ \cite[\S 6, Th{\'e}or{\`e}me 1]{Bourbaki_Algegre_4}. Now one sees that $J^c\coloneq J\cap\Ocal_{X,o}[\ybf]^{\Sigma_d}$ is a prime of height $d$ generated by linear polynomials (\ref{Vieta}), and the quotient domain of $J^c$ is 
        \[\Ocal_{X,o}[\ybf]^{\Sigma_d}/J^c\cong\Ocal_{X,o}.\]
        Hence, Going-Up Theorem for the integral extension $\Ocal_{X,o}[\ybf]^{\Sigma_d}\subseteq\Ocal_{X,o}[\ybf]$ implies that $J$ is also of height $d$.
        
        Since $\Ocal_{X,o}[\ybf]$ is a Cohen-Macaulay ring and $J$ is generated by $d$ elements (\ref{Vieta}), the Unmixedness Theorem \cite[Theorem 17.6]{Matsumura_1987} implies that $J$ is unmixed, i.e., $\Ass(\Ocal_{X,o}[\ybf]/J)$ are exactly the minimal primes containing $J$ in $\Ocal_{X,o}[\ybf]$. The minimal primes containing $J$ are exactly the primes lying over $J^c$.  Hence the Galois theory implies that $\Ass(\Ocal_{X,o}[\ybf]/J)$ are conjugate under the permutations on the variables $\ybf=(y_1,\dots,y_d)$. It is a folklore in field theory that the residue field $\kappa(Q)$ of $\Ocal_{X,o}[\ybf]$ at any $Q\in\Ass(\Ocal_{X,o}[\ybf]/J)$ is a splitting field of $\mu(t)$ over the fraction field $K$ of $\Ocal_{X,o}$. 
    \end{proof}

    Fix an associated prime $Q\in\Ass(\Ocal_{X,o}[\ybf]/J)$ and let 
    \[L=\kappa(Q)\]
    be the residue field of $\Ocal_{X,o}[\ybf]$ at $Q$. 
    Then $L$ is a finite Galois extension of $K=\Frac(\Ocal_{X,o})$. Since both $\Ocal_{X,o}$ and $\widehat{\Ocal_{X,o}}$ are regular local rings, they are also normal integral domains. Let $\overline{\Ocal_{X,o}}^L$ be the integral closure of $\Ocal_{X,o}$ in $L$. Since $\Ocal_{X,o}[\ybf]/Q$ is an integral extension of $\Ocal_{X,o}$ with a fraction field $L=\Frac(\Ocal_{X,o}[\ybf]/Q)$, $\overline{\Ocal_{X,o}}^L$ is also the integral closure $\overline{\Ocal_{X,o}[\ybf]/Q}$ of $\Ocal_{X,o}[\ybf]/Q$ in $L$. Then we have the following criterion, formulated in terms of the Jacobian matrix associated with any finite presentation of $\overline{\Ocal_{X,o}}^{L}$ over $\Ocal_{X,o}$.
    
    %Notice that the power series ring $\widehat{\Ocal_{X,o}}$ is also a regular local ring, we have the following  criteria,  formulated in terms of the Jacobian matrix associated with any finite presentation of $\overline{\Ocal_{X,o}}^{L}$ over $\Ocal_{X,o}$.

    \begin{theorem}\label{Thm:PowerSeriesCriteria4Jacobian}

 Let $\overline{\Ocal_{X,o}}^{L}=\Ocal_{X,o}[z_1,\dots,z_N]/Q'$ be a finite presentation of $\overline{\Ocal_{X,o}}^{L}$ over $\Ocal_{X,o}$, where $Q'$ is generated by $ g_1,\dots,g_D\in\Ocal_{X,o}[z_1,\dots,z_N]$. Then we necessarily have $D \geqslant N$.
 %Then necessarily $D\geqslant N$.
 Moreover, % $\overline{\Ocal_{X,o}}^{L}$ is unramified over $\Ocal_{X,o}$ at a maximal $\nfrak\subset\overline{\Ocal_{X,o}}^{L}$ (., 
        all roots of $\mu(t)$ belong to $\widehat{\Ocal_{X,o}}$ if and only if the Jacobian matrix $$\frac{\partial(f_1,\dots,f_D)}{\partial(z_1,\dots,z_N)}\mod\nfrak$$ 
        has full rank $N$ in $\kappa(\nfrak)^{D\times N}$.
    \end{theorem}
    In order to prove this theorem, we need to introduce some notations and lemmas.

    \begin{lemma}\label{Lem:PowerSeriesCriteria1}
        The following statements are equivalent:
        \begin{enumerate}[label=(\arabic*)]
            \item\label{Lem:PowerSeriesCriteria1.1} All roots of $\mu(t)$ over $K$ belong to $\widehat{\Ocal_{X,o}}$;

            \item\label{Lem:PowerSeriesCriteria1.2} There exists an $\Ocal_{X,o}$-algebra map $\Ocal_{X,o}[\ybf]/Q\rightarrow\widehat{\Ocal_{X,o}}$;
            
            \item\label{Lem:PowerSeriesCriteria1.3} There exists a $K$-algebra map $L\rightarrow\Frac(\widehat{\Ocal_{X,o}})$;
            
            \item\label{Lem:PowerSeriesCriteria1.4} There exists an $\Ocal_{X,o}$-algebra map $\overline{\Ocal_{X,o}[\ybf]/Q}\rightarrow\widehat{\Ocal_{X,o}}$.
        \end{enumerate}
        Moreover, if any statement above holds, then the algebra map therein is an injection. 
    \end{lemma}
    
    \begin{proof}
        \ref{Lem:PowerSeriesCriteria1.1}$\Rightarrow$\ref{Lem:PowerSeriesCriteria1.2}: If all the roots of $\mu(t)$ belong to $\widehat{\Ocal_{X,o}}$, then there is an $\Ocal_{X,o}$-algebra map $\varphi:\Ocal_{X,o}[\ybf]\rightarrow\widehat{\Ocal_{X,o}}$ sending $\ybf$ to the roots of $\mu(t)$ in $\widehat{\Ocal_{X,o}}$. The kernel of $\varphi$ lies over the zero ideal of $\Ocal_{X,o}$, which implies that $\ker\varphi$ conjugates with $Q$ under the permutations on $\ybf$. Hence, there is an $\Ocal_{X,o}$-algebra injection $\overline{\varphi}:\Ocal_{X,o}[\ybf]/Q\rightarrow\widehat{\Ocal_{X,o}}$.

        \ref{Lem:PowerSeriesCriteria1.2}$\Rightarrow$\ref{Lem:PowerSeriesCriteria1.1}: If there is an $\Ocal_{X,o}$-algebra map $\Ocal_{X,o}[\ybf]/Q\rightarrow\widehat{\Ocal_{X,o}}$, then the kernel lies over the zero ideal of $\Ocal_{X,o}$, which implies that $\Ocal_{X,o}[\ybf]/Q\rightarrow\widehat{\Ocal_{X,o}}$ is an injection, and the images of $\ybf~\mathrm{mod}~Q$ are the all roots of $\mu(t)$ over $K$. 

        \ref{Lem:PowerSeriesCriteria1.2}$\Rightarrow$\ref{Lem:PowerSeriesCriteria1.3}: The $\Ocal_{X,o}$-algebra map $\Ocal_{X,o}[\ybf]/Q\rightarrow\widehat{\Ocal_{X,o}}$ in (2) induces the fraction field extension $L\rightarrow\Frac(\widehat{\Ocal_{X,o}})$ over $K$.

        \ref{Lem:PowerSeriesCriteria1.3}$\Rightarrow$\ref{Lem:PowerSeriesCriteria1.1}: If there exists a $K$-algebra map $L\rightarrow\Frac(\widehat{\Ocal_{X,o}})$, then the images of $\ybf$ are the all roots of $\mu(t)$ over $K$.

        \ref{Lem:PowerSeriesCriteria1.3}$\Rightarrow$\ref{Lem:PowerSeriesCriteria1.4}: %Recall that $\Ocal_{X,o}$ and $\widehat{\Ocal_{X,o}}$ are respectively integrally closed in their fraction field $K$ and $\Frac(\widehat{\Ocal_{X,o}})$, hence the integral closure of $\Ocal_{X,o}$ in $\Frac(\widehat{\Ocal_{X,o}})$ is contained in $\widehat{\Ocal_{X,o}}$. 
        Recall that $\Ocal_{X,o}$ and $\widehat{\Ocal_{X,o}}$ are integrally closed in their respective fraction fields $K$ and $\Frac(\widehat{\Ocal_{X,o}})$. Consequently, the integral closure of $\Ocal_{X,o}$ in $\Frac(\widehat{\Ocal_{X,o}})$ is contained in $\widehat{\Ocal_{X,o}}$.
        By extension of the field $L\rightarrow\Frac(\widehat{\Ocal_{X,o}})$ over $K$ in \ref{Lem:PowerSeriesCriteria1.3}, the integral closure of $\overline{\Ocal_{X,o}}^L=\overline{\Ocal_{X,o}[\ybf]/Q}$ in $\Frac(\widehat{\Ocal_{X,o}})$ must also be contained in $\widehat{\Ocal_{X,o}}$.

        \ref{Lem:PowerSeriesCriteria1.4}$\Rightarrow$\ref{Lem:PowerSeriesCriteria1.2}: $\Ocal_{X,o}[\ybf]/Q\rightarrow\widehat{\Ocal_{X,o}}$ in \ref{Lem:PowerSeriesCriteria1.2} is obtained by combining $\Ocal_{X,o}[\ybf]/Q\rightarrow\overline{\Ocal_{X,o}[\ybf]/Q}$ and $\overline{\Ocal_{X,o}[\ybf]/Q}\rightarrow\widehat{\Ocal_{X,o}}$ in \ref{Lem:PowerSeriesCriteria1.4}.
    \end{proof}

    Notice that $\Frac(\widehat{\Ocal_{X,o}})$ is a considerably ``wild'' extension field of $K=\Frac(\Ocal_{X,o})$, since $\Frac(\widehat{\Ocal_{X,o}})$ is transcendental over $K$, while $L$ must be algebraic over $K$. Thus, we need an algebraic approximation of $\widehat{\Ocal_{X,o}}$ to refine the lemma \ref{Lem:PowerSeriesCriteria1}. Fortunately, there does indeed exist such a substitution, which is the Henselization $\Ocal_{X,o}^h$ of $\Ocal_{X,o}$ \cite[\S 18]{EGA_IV}\cite[Chapter 1]{Milne_1980}\cite[Chapter VII]{Nagata_1962}\cite{Raynaud_1970}. We also recommend the Stacks Project to readers for more information on Henselization as well as the ramification theory we will use subsequently. For convenience, we shall list the properties of the Henselization of $\Ocal_{X,o}$ to be used. First, we review some basic settings in ramification theory.  

\begin{notation}
Let $K = \Frac(\Ocal_{X,o})$.
\begin{enumerate}[label=(\Roman*), itemsep=1ex]
    \item Let $E$ be the separable closure of $K$, viewed as a subfield of a separable closure of $\Frac(\widehat{\Ocal_{X,o}})$, and denote the absolute Galois group of $K$ by
    \[
        G \coloneq \Gal(E|K).
    \]

    \item Let $\overline{\Ocal_{X,o}}^{E}$ be the integral closure of $\Ocal_{X,o}$ in $E$, and define
    \[
        \mfrak' \coloneq \widehat{\mfrak_{X,o}} \cap \overline{\Ocal_{X,o}}^{E},
    \]
    which is the maximal ideal of $\overline{\Ocal_{X,o}}^{E}$ lying over $\mfrak_{X,o}$.

    \item Let
    \[
        D \coloneq D(\mfrak'|\mfrak_{X,o}) 
        = \{\sigma \in G \mid \sigma(\mfrak') = \mfrak'\}
    \]
    be the decomposition group of $\mfrak'$ over $\mfrak_{X,o}$, and denote by $E^{D}$ its fixed field (the decomposition field of $K$).

    \item Let $\overline{\Ocal_{X,o}}^{D}$ be the integral closure of $\Ocal_{X,o}$ in $E^{D}$, and set
    \[
        (\mfrak')^{D} \coloneq \mfrak' \cap E^{D},
    \]
    which is the maximal ideal of $\overline{\Ocal_{X,o}}^{D}$.
\end{enumerate}
\end{notation}

\begin{comment}
    \begin{notation}
    Let $K=\Frac(\Ocal_{X,o})$.
    \begin{enumerate}[label=(\Roman*), itemsep=1ex]
        \item Let $E$ be the separable closure of $K$ embedded in a separable closure of $\Frac(\widehat{\Ocal_{X,o}})$.
        
        \item Let $G=\Gal(E|K)$ be the absolute Galois group of $K$.

        \item Let $\overline{\Ocal_{X,o}}^E$ be the integral closure of $\Ocal_{X,o}$ in $E$.

        \item Let $\mfrak'=\widehat{\mfrak_{X,o}}\cap\overline{\Ocal_{X,o}}^E$ be the maximal ideal of $\overline{\Ocal_{X,o}}^E$ lying over $\mfrak_{X,o}$.

        \item Let $D\coloneq D(\mfrak'|\mfrak_{X,o})=\{\sigma\in G : \sigma(\mfrak')=\mfrak'\}$ be the decomposition group of $\mfrak'$ over $\mfrak_{X,o}$.

        \item Let $E^D$ be the fixed subfield of $E$ under $D$ (the decomposition field of $K$).

        \item Let $\overline{\Ocal_{X,o}}^D$ be the integral closure of $\Ocal_{X,o}$ in $E^D$.

        \item Let $(\mfrak')^D\coloneq \mfrak'\cap E^D$ be the maximal ideal of $\overline{\Ocal_{X,o}}^D$.
    \end{enumerate}
    \end{notation}
\end{comment}

    Now we can state the properties of the Henselizaion.

    \begin{property}
      [properties of Henselization]\cite{EGA_IV,Milne_1980,Nagata_1962,Raynaud_1970}\label{lem:PropertesOfHenselization}\hfill
        \begin{enumerate}[label=(\roman*)]

            \item\label{lem:PropertesOfHenselization1} $\Ocal_{X,o}^h$ is a Henselian local extension of $\Ocal_{X,o}$ with maximal ideal $\mfrak_{X,o}^h=\mfrak_{X,o}\Ocal_{X,o}^h$ and residue field $\kappa(\mfrak_{X,o}^h)=\kappa(\mfrak_{X,o})$;
            %\item\label{lem:PropertesOfHenselization1} $\Ocal_{X,o}^h$ is a local ring extension of $\Ocal_{X,o}$, which is Henselian and of the maximal ideal $\mfrak_{X,o}^h=\mfrak_{X,o}\Ocal_{X,o}^h$ and the residue field $\kappa(\mfrak_{X,o}^h)=\kappa(\mfrak_{X,o})$;
            
            \item\label{lem:PropertesOfHenselization2} $\Ocal_{X,o}^h$ is the local ring (i.e., the direct limit) regarding to the {\'etale} neighborhoods of $o$ in $X$ with trivial residue field extension; 
            
            %\item\label{lem:PropertesOfHenselization2} $\Ocal_{X,o}^h$ is the direct limit of the system of étale neighborhoods of $o$ in $X$ with trivial residue field extension;
            \item\label{lem:PropertesOfHenselization3} $\Ocal_{X,o}^h$ is a regular local ring of dimension $m=\dim\Ocal_{X,o}$;
            
            \item\label{lem:PropertesOfHenselization4} The canonical series of $\Ocal_{X,o}$-algebra bijections $$\Ocal_{X,o}/(\mfrak_{X,o})^e\xrightarrow{\sim}\Ocal_{X,o}^h/(\mfrak_{X,o}^h)^e,~\text{for}~e\in\ZZ,$$ 
            induces the $\Ocal_{X,o}$-algebra injection $$\Ocal_{X,o}^h\rightarrow\widehat{\Ocal_{X,o}}$$ 
            and the $K$-algebra injection $\Frac(\Ocal_{X,o}^h)\rightarrow\Frac(\widehat{\Ocal_{X,o}})$;
            
            \item\label{lem:PropertesOfHenselization5} $\Ocal_{X,o}^h$ is isomorphic to the local ring $(\overline{\Ocal_{X,o}}^D)_{(\mfrak')^D}$ as an $\Ocal_{X,o}$-algebra, and its fraction field $\Frac(\Ocal_{X,o}^h)$ is isomorphic to $E^D$ as a $K$-algebra. In the sequel, we identify  $\Ocal_{X,o}^h$ and $\Frac(\Ocal_{X,o}^h)$ with $(\overline{\Ocal_{X,o}}^D)_{(\mfrak')^D}$ and $E^D$, respectively.

            \item\label{lem:PropertesOfHenselization6} (Specical case of Artin's Approximation \cite{MichaelArtin1969}) The equalities hold
            \begin{itemize}
                \item $\overline{\Ocal_{X,o}}^E\cap\widehat{\Ocal_{X,o}}=\Ocal_{X,o}^h$,

                \item $E\cap\widehat{\Ocal_{X,o}}=\Ocal_{X,o}^h$,  

                \item $\overline{\Ocal_{X,o}}^E\cap\Frac(\widehat{\Ocal_{X,o}})=\Ocal_{X,o}^h$, and
                
                \item $E\cap\Frac(\widehat{\Ocal_{X,o}})=\Frac(\Ocal_{X,o}^h)$.
            \end{itemize}
        \end{enumerate}
      %  \end{property}
    \end{property}
    
    %Roughly speaking, M. Artin's approximation theorem \cite[Theorem (1.10)]{MichaelArtin1969} implies that given any polynomial system $\Scal$ in finite variables with coefficients in the Henselization $\Ocal_{X,o}^h$, any solution of $\Scal$ in $\widehat{\Ocal_{X,o}}$ can be approximated by a solution of $\Scal$ in $\Ocal_{X,o}^h$ in arbitrary given precise. As a corollary, if $S$ has only finite solutions in $\Ocal_{X,o}^h$, then any solution of $\Scal$ in $\widehat{\Ocal_{X,o}}$ is actually in $\Ocal_{X,o}^h$. 
   % Using Properties \ref{lem:PropertesOfHenselization}, the following refinement of Lemma \ref{Lem:PowerSeriesCriteria1} holds. 

    Using Property~\ref{lem:PropertesOfHenselization}, we obtain the following refinement of Lemma~\ref{Lem:PowerSeriesCriteria1}.
    
    \begin{lemma}[continuing Lemma \ref{Lem:PowerSeriesCriteria1}] \label{Lem:PowerSeriesCriteria2}
        The following statements are equivalent to the statements in Lemma \ref{Lem:PowerSeriesCriteria1}:
        \begin{enumerate}[label=(\arabic*')]
           % \item All the roots of $\mu(t)$ over $K$ belong to $\widehat{\Ocal_{X,o}}$;

            \item\label{Lem:PowerSeriesCriteria2.1} All roots of $\mu(t)$ over $K$ belong to $\Ocal_{X,o}^h$;

            \item\label{Lem:PowerSeriesCriteria2.2} There exists an $\Ocal_{X,o}$-algebra map $\Ocal_{X,o}[\ybf]/Q\rightarrow\Ocal_{X,o}^h$;
            
            \item\label{Lem:PowerSeriesCriteria2.3} There exists a $K$-algebra map $L\rightarrow\Frac(\Ocal_{X,o}^h)$;
            
            \item\label{Lem:PowerSeriesCriteria2.4} There exists an $\Ocal_{X,o}$-algebra map $\overline{\Ocal_{X,o}[\ybf]/Q}\rightarrow\Ocal_{X,o}^h$.
        \end{enumerate}
        Moreover, if any of the above statements holds, the associated algebra map is injective.
        
        %Moreover, if any statement above holds, then the algebra map therein is an injection.
    \end{lemma}
    
    \begin{proof}
        %There is a canonical $\Ocal_{X,o}$-algebra map $\Ocal_{X,o}^h\rightarrow\widehat{\Ocal_{X,o}}$ induced by the canonical $\Ocal_{X,o}$-algebra isomorphism $$\Ocal_{X,o}/\mfrak_{X,o}^e\cong\Ocal_{X,o}^h/(\mfrak_{X,o}\Ocal_{X,o}^h)^e\cong\widehat{\Ocal_{X,o}}/(\mfrak_{X,o}\widehat{\Ocal_{X,o}})^e$$ for any integer $e$. Since $\Ocal_{X,o}$ is a regular local domain of dimension $m$, so also is $\Ocal_{X,o}^h$, which implies that the $\Ocal_{X,o}$-algebra map $\Ocal_{X,o}^h\rightarrow\widehat{\Ocal_{X,o}}$ is an injection. 
        Combine the $\Ocal_{X,o}$-algebra maps in \ref{Lem:PowerSeriesCriteria2.2} and \ref{Lem:PowerSeriesCriteria2.4} with the $\Ocal_{X,o}$-algebra map $\Ocal_{X,o}^h\rightarrow\widehat{\Ocal_{X,o}}$ in Property \ref{lem:PropertesOfHenselization} \ref{lem:PropertesOfHenselization4}, one sees that \ref{Lem:PowerSeriesCriteria2.1}$\Rightarrow$\ref{Lem:PowerSeriesCriteria1.1}, \ref{Lem:PowerSeriesCriteria2.2}$\Rightarrow$\ref{Lem:PowerSeriesCriteria1.2}, \ref{Lem:PowerSeriesCriteria2.3}$\Rightarrow$\ref{Lem:PowerSeriesCriteria1.3}, and \ref{Lem:PowerSeriesCriteria2.4}$\Rightarrow$\ref{Lem:PowerSeriesCriteria1.4}; Property \ref{lem:PropertesOfHenselization} \ref{lem:PropertesOfHenselization5} and \ref{lem:PropertesOfHenselization6} imply that the roots of $\mu(t)$ in $\widehat{\Ocal_{X,o}}$ are all in $\Ocal_{X,o}^h$, and the maps in \ref{Lem:PowerSeriesCriteria1.1}, \ref{Lem:PowerSeriesCriteria1.2}, \ref{Lem:PowerSeriesCriteria1.3} and \ref{Lem:PowerSeriesCriteria1.4} factor though the corresponding maps in \ref{Lem:PowerSeriesCriteria2.1}, \ref{Lem:PowerSeriesCriteria2.2}, \ref{Lem:PowerSeriesCriteria2.3} and \ref{Lem:PowerSeriesCriteria2.4}, respectively. Hence \ref{Lem:PowerSeriesCriteria1.1}$\Rightarrow$\ref{Lem:PowerSeriesCriteria2.1}, \ref{Lem:PowerSeriesCriteria1.2}$\Rightarrow$\ref{Lem:PowerSeriesCriteria2.2}, \ref{Lem:PowerSeriesCriteria1.3}$\Rightarrow$\ref{Lem:PowerSeriesCriteria2.3}, and \ref{Lem:PowerSeriesCriteria1.4}$\Rightarrow$\ref{Lem:PowerSeriesCriteria2.4}.
    \end{proof}

    Note that $\Frac(\Ocal_{X,o}^h)=E^D$ is in general an infinite algebraic extension field of $K=\Frac(\Ocal_{X,o})$, the criteria in Lemma \ref{Lem:PowerSeriesCriteria2} are still incapable in the sense of computation. Hence, we need criteria that describe the field extension $K\subseteq L$ or the integral extension $\Ocal_{X,o}\subseteq\overline{\Ocal_{X,o}}^L$ more explicitly. To establish the next part of the criteria for the total split of $\mu(t)$ in $\widehat{\Ocal_{X,o}}$, we need to review several types of morphism in algebraic geometry.

\begin{notation}
Let $\Afrak$ and $\Bfrak$ be Noetherian commutative rings with unity, and let
\[
\phi \colon \Afrak \to \Bfrak
\]
be a ring map. Let $\pfrak \subset \Afrak$ and $\qfrak \subset \Bfrak$ be primes such that
\[
\pfrak = \phi^{-1}(\qfrak).
\]
%We use the following terminology for $\phi$, 
The ring map $\phi$ is called
\begin{enumerate}[label=(\Alph*), itemsep=1ex]
    \item \emph{of finite type}, if $\phi$ factors as
    \[
    \Afrak \hookrightarrow \Afrak[z_1,\dots,z_N] \twoheadrightarrow \Bfrak
    \]
    for some $N \in \mathbb{N}$ and variables $z_1,\dots,z_N$;

    \item \emph{finite}, if $\Bfrak$ is a finite $\Afrak$-module;

    \item \emph{flat}, if $\Bfrak$ is a flat $\Afrak$-module;

    \item \emph{flat at $\qfrak$}, if $\Bfrak_{\qfrak}$ is a flat $\Afrak_{\pfrak}$-module;

    \item \emph{unramified at $\qfrak$}, if $\phi$ is of finite type,
    \[
    \phi(\pfrak)\Bfrak_{\qfrak} = \qfrak \Bfrak_{\qfrak},
    \]
    and the induced residue field extension $\kappa(\qfrak)/\kappa(\pfrak)$ is finite separable;

    \item \emph{unramified}, if $\phi$ is of finite type and unramified at every prime $\Qfrak \subset \Bfrak$;

    \item \emph{\'etale at $\qfrak$}, if $\phi$ is flat at $\qfrak$ and unramified at $\qfrak$;

    \item \emph{\'etale}, if $\phi$ is \'etale at every prime $\Qfrak \subset \Bfrak$.
\end{enumerate}
\end{notation}

    Now we can state and prove the following criteria.

    \begin{lemma}[continuing Lemma \ref{Lem:PowerSeriesCriteria1} and \ref{Lem:PowerSeriesCriteria2}]\label{Lem:PowerSeriesCriteria3}
        The following statements are equivalent to the statements in Lemma \ref{Lem:PowerSeriesCriteria1} and Lemma \ref{Lem:PowerSeriesCriteria2}:
        \begin{enumerate}
            %\item All the roots of $\mu(t)$ over $K$ belong to $\widehat{\Ocal_{X,o}}$;

            \item[(u')]\label{u'} $\Ocal_{X,o}\rightarrow\overline{\Ocal_{X,o}}^{L}$ is unramified at some maximal $\nfrak\subset\overline{\Ocal_{X,o}}^{L}$;

            \item[(u)] $\Ocal_{X,o}\rightarrow\overline{\Ocal_{X,o}}^{L}$ is unramified;

            \item[(e')] $\Ocal_{X,o}\rightarrow\overline{\Ocal_{X,o}}^{L}$ is {\'e}tale at some maximal $\nfrak\subset\overline{\Ocal_{X,o}}^{L}$;

            \item[(e)] $\Ocal_{X,o}\rightarrow\overline{\Ocal_{X,o}}^{L}$ is {\'e}tale;

            %\item[(7)] $\Ocal_{X,o}\rightarrow\overline{\Ocal_{X,o}}^{L}$ is smooth at some maximal $\nfrak\subset\overline{\Ocal_{X,o}}^{L}$;

            %\item[(7')] $\Ocal_{X,o}\rightarrow\overline{\Ocal_{X,o}}^{L}$ is smooth.
        \end{enumerate}
    \end{lemma}

    \begin{proof}
        (u)$\Rightarrow$(u') and (e)$\Rightarrow$(e') are clear, since it is a local property for a ring map to be unramified or to be {\'e}tale. 
        
        (e')$\Rightarrow$(u') and (e)$\Rightarrow$(u) are well known, since being (locally) {\'e}tale is equivalent to being both (locally) flat and (locally) unramified. 
        
        (u')$\Rightarrow$(e') and (u)$\Rightarrow$(e) are implied by \cite[Chapter 1, Theorem 3.20]{Milne_1980}.

        (u')$\Rightarrow$(u) and (e')$\Rightarrow$(e): Since $\Gal(L|K)$ transitively acts on the maximals of $\overline{\Ocal_{X,o}}^{L}$, being unramified or being {\'e}tale at any maximal will, respectively, imply the unramifiedness or the {\'e}taleness at every maximal. 

        \ref{Lem:PowerSeriesCriteria2.3}$\Rightarrow$(u'): %We shall show (3')$\Rightarrow$(u'). 
        Let $(\mfrak')^L=(\mfrak'\cap E^D)\cap L$ be the maximal of $\overline{\Ocal_{X,o}}^{L}$. \ref{Lem:PowerSeriesCriteria2.3} implies that $L$ is fixed by $D=D(\mfrak'|\mfrak_{X,o})$. Hence we have $D\subseteq\Gal(E|L)$ and $D(\mfrak'|(\mfrak')^L)=D$.
        
         The canonical short exact sequence of decomposition groups for Galois extensions $E|L|K$ %with maximals $\mfrak'|(\mfrak')^L|\mfrak_{X,o}$
        \[1 \rightarrow D(\mfrak'|(\mfrak')^L) \xrightarrow{\Id} D(\mfrak'|\mfrak_{X,o}) \rightarrow D((\mfrak')^L|\mfrak_{X,o}) \rightarrow 1\]
 implies that $D((\mfrak')^L|\mfrak_{X,o})=1$, whence the subfield of $L$ fixed by $D((\mfrak')^L|\mfrak_{X,o})$ is $L$ itself. According to  \cite[Theorem 41.2]{Nagata_1962}, the maximal ideal $\mfrak_{X,o}$ of $\Ocal_{X,o}$ generates the maximal ideal of the local ring $(\overline{\Ocal_{X,o}}^{L})_{(\mfrak')^L}$, and the extension of the residue field $\kappa((\mfrak')^L)|\kappa(\mfrak_{X,o})$ is trivial. Since $\Ocal_{X,o}$ is Noetherian and $L|K$ is finite separable, the integral closure $\overline{\Ocal_{X,o}}^{L}$ of $\Ocal_{X,o}$ in $L$ is well-known to be a finite $\Ocal_{X,o}$-module \cite[Corollary 10.16]{Nagata_1962}, which implies that $\overline{\Ocal_{X,o}}^{L}$ is of finite type over $\Ocal_{X,o}$. Therefore, $\Ocal_{X,o}\rightarrow\overline{\Ocal_{X,o}}^{L}$ is unramified at $(\mfrak')^D$.

        (e')$\Rightarrow$\ref{Lem:PowerSeriesCriteria2.4}: Note that the residue fields extension $\kappa(\nfrak)|\kappa(\mfrak_{X,o})$ is trivial since $\kappa(\nfrak)|\kappa(\mfrak_{X,o})$ is finite separable and $\kappa(\mfrak_{X,o})\cong\CC$ is separably closed. Hence $(\overline{\Ocal_{X,o}}^{L},\nfrak)$ is an {\'e}tale neighborhood of $(\Ocal_{X,o},\mfrak_{X,o})$ with trivial residue field extension. Now Property \ref{lem:PropertesOfHenselization} \ref{lem:PropertesOfHenselization2} implies that there is a canonical $\Ocal_{X,o}$-algebra map $\overline{\Ocal_{X,o}}^{L}\rightarrow\Ocal_{X,o}^h$ satisfying $\mfrak_{X,o}^h\cap\overline{\Ocal_{X,o}}^{L}=\nfrak$, which verifies that (e')$\Rightarrow$\ref{Lem:PowerSeriesCriteria2.4}.  
    \end{proof}

    %The lemma \ref{Lem:PowerSeriesCriteria3} induces the following criterion based on the Jacobian matrix of any finite presentation of $\overline{\Ocal_{X,o}}^{L}$ over $\Ocal_{X,o}$. 

    %The Lemma~\ref{Lem:PowerSeriesCriteria3} implies the following criterion, formulated in terms of the Jacobian matrix associated with any finite presentation of $\overline{\Ocal_{X,o}}^{L}$ over $\Ocal_{X,o}$.

    Now using Lemma \ref{Lem:PowerSeriesCriteria1}-\ref{Lem:PowerSeriesCriteria3} we can give the proof for Theorem \ref{Thm:PowerSeriesCriteria4Jacobian}. 

    \begin{proof}[Proof of Theorem \ref{Thm:PowerSeriesCriteria4Jacobian}]
        The inequality $D\geqslant N$ follows from the dimension inequality and the fact that $\overline{\Ocal_{X,o}}^{L}$ is an integral extension of $\Ocal_{X,o}$. Note that Lemma  \ref{Lem:PowerSeriesCriteria1}-\ref{Lem:PowerSeriesCriteria3} have established (1) $\Leftrightarrow$ (u'). Denote $\Afrak=\Ocal_{X,o}$ and $\Bfrak=\overline{\Ocal_{X,o}}^{L}$.
         Then $\overline{\Ocal_{X,o}}^{L}$ is unramified over $\Ocal_{X,o}$ at $\nfrak$ if and only if the module of Kähler differentials $\Omega_{\Bfrak/\Afrak}$ vanishes at $\nfrak$. This is equivalent to the condition that the canonical $\Bfrak$-module map
\begin{equation*}
    Q'/(Q')^2 \longrightarrow \Omega_{\Afrak[z_1,\dots,z_N]/\Afrak}\otimes_{\Afrak}\Bfrak, \quad f \mapsto \sum_{j=1}^N \frac{\partial f}{\partial z_j}\drm z_j
\end{equation*}
is surjective at $\nfrak$.
Finally, this holds if and only if the Jacobian matrix, which defines the linear map
%defined by the Jacobian matrix
%if and only if the Jacobian map defined by
\begin{equation*}
\begin{aligned}
    J: \bigoplus_{k=1}^D \Bfrak f_k &\longrightarrow \bigoplus_{j=1}^N \Bfrak \drm z_j = \Omega_{\Afrak[\zbf]/\Afrak}\otimes_{\Afrak}\Bfrak, \\
    (b_k f_k)_k &\longmapsto \sum_{j=1}^{N} \left( \sum_{k=1}^{D} b_k \frac{\partial f_k}{\partial z_j} \right) \drm z_j,
\end{aligned}
\end{equation*}
has full rank $N$ (equivalently, is surjective) at $\nfrak$.
\end{proof}
 
\begin{comment}
 Then $\overline{\Ocal_{X,o}}^{L}$ is unramified over $\Ocal_{X,o}$ if and only if the module $\Omega_{\Bfrak/\Afrak}$ of K{\"a}hler differentials of $\overline{\Ocal_{X,o}}^{L}$ over $\Ocal_{X,o}$ is zero at $\nfrak$, if and only if the canonical $\Bfrak$-module map 
       \begin{equation*}
Q'/(Q')^2\rightarrow\Omega_{\Afrak[z_1,\dots,z_N]/\Afrak}\otimes_{\Afrak}\Bfrak,~f \mapsto \sum_k\frac{\partial f}{\partial z_k}\drm z_k
        \end{equation*}
        is surjective at $\nfrak$, if and only if the Jacobian matrix 
        \begin{equation*}
            \frac{\partial(f_1,\dots,f_D)}{\partial(z_1,\dots,z_N)}:\bigoplus_{k=1}^D{\Bfrak f_k}\rightarrow\bigoplus_{j=1}^N{\Bfrak\drm z_k}
        \end{equation*} 
        is of full rank $N$ (equivalently, surjective) at $\nfrak$.  
\end{comment}

   Let $\mu(t)\in\Ocal_{X,o}[t]$ be a monic square-free polynomial. 
    We present Algorithm \ref{alg:factor} to determine whether    $\mu(t)\in\Ocal_{X,o}[t]$ is completely split over $\widehat{\mathcal{O}_{X,o}}$.
    
    \begin{algorithm}[htbp]
\caption{Determine Whether a Polynomial is Completely split   over \(\widehat{\mathcal{O}_{X,o}}\)} \label{alg:factor}
    
		%\caption{Algorithm for Whether a Polynomial Totally Splits on \(\widehat{\Ocal_{X,o}}\)}

     %   \caption{Criterion for the Total Splitting of Polynomials over \(\widehat{\mathcal{O}_{X,o}}\)}

%\KwIn{Defining ideal $\mathbf{I}(X) \subseteq \CC[\xbf]$ of a variety $X$ (with a regular point  $o \in X$); a monic polynomial $\mu(t) \in \Ocal_{X,o}[t]$ of degree $d$, represented by coefficients $c_k = F_k/G_k$ where $G_k(o) \neq 0$.}

  \KwIn{Polynomial ring $\mathbb{C}[\xbf]$ with the defining ideal $\mathbf{I}(X)\subseteq\xbf\CC[\xbf]$ of the variety $X$,
        a monic polynomial \(\mu(t)=t^d+c_1t^{d-1}+\cdots+c_d\), %\(c_k\in\Ocal_{X,o}\), 
        \(c_k=F_k/G_k\mod\Ibf(X)\), $F_k, G_k\in\mathbb{C}[\xbf]$ and \(G_k(o)\neq0\). }
      %  a monic polynomial $\mu(t) \in \Ocal_{X,o}[t]$ of degree $d$, represented by coefficients $c_k = F_k/G_k$ where $G_k(o) \neq 0$.}
\KwOut{\textbf{True} if $\mu(t)$ splits completely over $\widehat{\mathcal{O}_{X,o}}$ (i.e., $\mu(t)=\prod_{k=1}^d(t-\lambda_k)$ for some $\lambda_k \in \widehat{\mathcal{O}_{X,o}}$), together with defining ideals $Q', \nfrak \subseteq \CC[\xbf,\ybf,\wbf]$; otherwise \textbf{False}.}
        
       % \KwIn{Polynomial ring $\mathbb{C}[\xbf]$ with the defining ideal $\mathbf{I}(X)\subseteq\xbf\CC[\xbf]$ of the variety $X$,
        %a monic polynomial \(\mu(t)=t^d+c_1t^{d-1}+\cdots+c_d\), %\(c_k\in\Ocal_{X,o}\), 
       % \(c_k=F_k/G_k\mod\Ibf(X)\), $F_k, G_k\in\mathbb{C}[\xbf]$ and \(G_k(o)\neq0\). 
      %  a monic polynomial $\mu(t) \in \Ocal_{X,o}[t]$ of degree $d$, represented by coefficients $c_k = F_k/G_k$ where $G_k(o) \neq 0$.}

		%\KwOut{\textbf{True} if there exist \(\lambda_k\in\widehat{\mathcal{O}_{X,0}},k=1,\dots,d\) satisfying \(\mu=\prod_{k=1}^d(t-\lambda_k)\), ideals \(Q',\nfrak\) in \(\CC[\xbf,\ybf,\wbf]\); \textbf{False} otherwise. }
        
        \SetAlgoLined
		\begin{enumerate}[label=\arabic*.]\upshape
			\item Define the ideal
            \[J \coloneq \langle\, e_{k}(\mathbf{y}) \cdot G_{k} - (-1)^{k} F_{k} \mid k=1,\dots,d \,\rangle \subset \mathbb{C}[\mathbf{x}, \mathbf{y}],\]
            where the \(e_{k}\) denote the elementary symmetric polynomials in the variables \(\mathbf{y}=(y_{1},\dots,y_{d})\).
            \item Select an associated prime $Q\in\Spec(\CC[\xbf,\ybf])$ of the ideal \(J+\Ibf(X)\CC[\xbf,\ybf]\) satisfying \(Q\cap\CC[\xbf]\subseteq\xbf\CC[\xbf]\).
            %\item Take the associated primes of the ideal \(\Ibf(X)+J\): $\Ass(\CC[\xbf,\ybf]/(\Ibf(X)+J))=\{Q_1,\dots,Q_s\}$.
            %\item Choose one \(Q_j\) in \(Q_1,\dots,Q_s\) satisfying that \(Q_j\cap\CC[\xbf]\subseteq\xbf\CC[\xbf]\). Denote by \(Q=Q_j\). 
            %Find the generators \(g_1,\dots,g_D\) of \(Q\).

\item Compute an affine presentation of the integral closure of $\CC[\xbf,\ybf]/Q$ in its fraction field, denoted by $\CC[\xbf,\ybf,\wbf]/Q'$, introducing new variables $\wbf=(w_1,\dots,w_e)$.  Compute a generating set $g_1,\dots,g_D$ for the ideal $Q'$, observing that necessarily $D \geqslant d+e$.   
         %   \item Compute an affine presentation of the integral closure of the quotient ring $\CC[\xbf,\ybf]/Q$ in its fraction field, say $\CC[\xbf,\ybf,\wbf]/Q'$, with new variables $\wbf=(w_1,\dots,w_e)$. Compute a list of generators \(g_1,\dots,g_D\) for $Q'$; here necessarily \(D \geqslant d+e\).

            \item Find an associated prime $\mathfrak{n}\in\Spec(\CC[\xbf,\ybf,\wbf])$ of the ideal $Q'+\xbf\CC[\xbf,\ybf,\wbf]$ that satisfies $\mathfrak{n}\cap\CC[\xbf]\subseteq\xbf\CC[\xbf]$. 
            Note that any such prime $\mathfrak{n}$ is necessarily a maximal ideal.
           % Such $\mathfrak{n}$ is necessarily a maximal ideal of $\CC[\xbf,\ybf,\wbf]$.
           
            \item Calculate the Jacobian matrix of \(g_1,\dots,g_D\) with respected to \(\ybf,\wbf\) modulo \(\nfrak\): 
            \[
                \frac{\partial(g_1,\dots,g_D)}{\partial(\ybf,\wbf)}(\mod\nfrak)\in(\CC[\xbf,\ybf,\wbf]/\nfrak)^{D\times (d+e)}.
            \]
If the Jacobian matrix modulo $\nfrak$ has full rank $d+e$, return \textbf{True} and  the ideals $Q'$, $\nfrak$; otherwise, return \textbf{False}.
            
            %If the Jacobian matrix (\(\mod\nfrak\)) is of full rank $(d+e)$, return \textbf{True} and ideals \(Q',\nfrak\); otherwise, return \textbf{False}.
		\end{enumerate}
	\end{algorithm}

    \begin{remark}
    By the Implicit Function Theorem, we can apply Newton-Hensel iteration to the functions $g_1,\dots,g_D$ using the full-rank Jacobian matrix from Step~5. 
    This enables the computation of the power series expansions of $\ybf$ and $\wbf$ modulo $Q'$ to any desired precision.

       % According to the Implicit Function Theorem, one can apply the Newton-Hensel iteration on $g_1,\dots,g_D$ and their full-rank Jacobian matrix in step 5 to compute the power series $\ybf,\wbf\mod Q'$ in any given precise. 
    \end{remark}

    \begin{remark}
    In practice, if the variety $X$ is defined over a subfield $F \subseteq \CC$ (such as $\QQ$ or a number field), Algorithm~\ref{alg:factor} can be executed entirely over $F$. Specifically, one can replace $\CC$ with $F$ in the coefficient fields of the polynomial rings and compute all associated primes over $F$. This adaptation is justified by the following theoretical parallels:
    \begin{enumerate}
        \item \textbf{Regularity descends:} For an $F$-variety $Y$ and a closed point $p \in Y$, $p$ is a regular point if and only if the points in the base change $Y_{\CC}$ lying over $p$ are regular.
        \item \textbf{Algebraic Generalization:} The results built on $\Ocal_{X,o}$ extend to $\Ocal_{Y,p}$ by replacing the Henselization $\Ocal_{X,o}^h$ with the \textit{strict} Henselization $\Ocal_{Y,p}^{sh}$.
        \item \textbf{Ramification Theory:} The role of the decomposition group $D(\mfrak'|\mfrak_{X,o})$ is assumed by the inertia group $I(\mfrak'|\mfrak_{Y,p})$.
    \end{enumerate}
    Consequently, one can develop analogues of Lemmas~\ref{Lem:PowerSeriesCriteria1}--\ref{Lem:PowerSeriesCriteria3} in this arithmetic setting, ultimately recovering the same Jacobian criterion (Theorem~\ref{Thm:PowerSeriesCriteria4Jacobian}) over the field $F$.    
    \end{remark}

    %In practice, if the defining polynomials of $X$ are over a subfield $F$ of $\CC$, say the rational field $\QQ$ or any number field (a finite extension of $\QQ$), then the coefficient field of the polynomial rings in Algorithm \ref{alg:factor} can be substituted with $F$ rather than $\CC$, and all the associated prime can be correspondingly taken over the field $F$. The reason is as follows. First of all, for an $F$-variety $Y$ with a closed point $p\in Y$, $p$ is a regular point of $Y$ if and only if the preimages of $p$ via the base change $Y_{\CC}\rightarrow Y$ are regular points of $Y_{\CC}$. Next, by replacing $\Ocal_{X,o}$ with $\Ocal_{Y,p}$, the Henselization $\Ocal_{X,o}^h$ with the strict Henselization $\Ocal_{X,o}^{sh}$, the decomposition group $D(\mfrak'|\mfrak_{X,o})$ with the inertia group $I(\mfrak'|\mfrak_{Y,p})$, one can parallel develop Lemma \ref{Lem:PowerSeriesCriteria1} and Lemma \ref{Lem:PowerSeriesCriteria2}, Lemma \ref{Lem:PowerSeriesCriteria3} using the ramification theory on $Y$, and eventually obtain the same Theorem \ref{Thm:PowerSeriesCriteria4Jacobian}. 

We present two examples to illustrate the Algorithm \ref{alg:factor}. All computations are performed in \textsf{Macaulay2} over the field \(\QQ\).

    \begin{example} 
        The minimal polynomial of the matrix
        \begin{equation*}
            A=\begin{pmatrix}
                -x_1 & x_2 & 0\\
                x_2 & -x_1 & \sqrt{2}x_2\\
                0 & \sqrt{2}x_2 &2x_1 
            \end{pmatrix}
        \end{equation*}
        is \(\mu_A(t)=t^3-3(x_1^2+x_2^2)t-2x_1^3\). This is the  hyperbolic polynomial given in \cite[Example 5.7]{MR2372149}. 
        Its \emph{Vieta's ideal} is
        \[        J=\ideal{y_{1}+y_{2}+y_{3},y_{1}y_{2}+y_{1}y_{3}+y_{2}y_{3}+3x_{1}^{2}+3x_{2}^{2},y_{1}y_{2}y_{3}-2x_{1}^{3}},
        \]
        which is prime in \(\QQ[\xbf,\ybf]\), so \(Q=J\).

        The integral closure of \(\QQ[\xbf,\ybf]/Q\) is given by \[\QQ[\xbf,\ybf,\wbf]/Q',\quad\wbf=(w_{0,0},w_{0,1}),\]
        where \(Q'=\<y_{1}+y_{2}+y_{3}, \  y_{2}^{2} +y_{2}y_{3}+y_{3}^{2}-3 x_{1}^{2}-3x_{2}^{2},\ w_{0,1}x_{2}-y_{2}y_{3}-y_{2}x_{1}-y_{3}x_{1}-x_{1}^{2},\ w_{0,1}y_{2}+w_{0,1}y_{3}-w_{0,1}x_{1}-3x_{1}x_{2},\ w_{0,0}x_{2}-y_{3}^{2}+y_{3}x_{1}+2x_{1}^{2},\ w_{0,0}y_{3}+w_{0,0}x_{1}-3y_{3}x_{2},\ w_{0,0}y_{2}+w_{0,0}x_{1}-w_{0,1}y_{3}+2w_{0,1}x_{1},\ w_{0,1}^{2}+3y_{2}y_{3}-3x_{1}^{2},\ w_{0,0}w_{0,1}-3y_{2}y_{3}-3y_{3}x_{1},\ w_{0,0}^{2}-3y_{3}^{2}+6y_{3}x_{1}\>\).

The Jacobian matrix of generators of \(Q'\) with respect to variables\(\ybf,\wbf\) at the point \(\nfrak\) is 
        \[\begin{pmatrix}
            0&0&0&0&0&0&0&0&0&0\\
            0&0&0&0&0&0&0&0&0&0\\
            1&0&0&0&0&0&0&0&0&0\\
            1&0&0&0&0&0&0&0&0&0\\
            1&0&0&0&0&0&0&0&0&0
        \end{pmatrix},\]
        which has rank \(1\). 
        Consequently, $\mu_A$ \emph{does not} split over $\widehat{\Ocal_{X,o}}$, and thus $A$ is not diagonalizable.
    \end{example}

%The following examples, computed using \textsf{Macaulay2}, illustrate the steps of Algorithm~\ref{alg:factor} and verify its correctness.

The following example, originally due to Rellich~\cite{rellich1937storungstheorie}, is taken from the survey by Parusi\'nski and Rainer~\cite[Example 2.8]{MR4890428}.

    \begin{example}\label{exp:rellich}\cite{rellich1937storungstheorie}
        Given the matrix 
        \begin{equation*}
            A=\begin{pmatrix}
                2x_1 & x_1+x_2 \\
                x_1+x_2 & 2x_2
            \end{pmatrix},
        \end{equation*}
        its minimal polynomial is
        \[\mu_A(t)=t^2+(-2x_1-2x_2)t-x_1^2+2x_1x_2-x_2^2.\]
The corresponding Vieta's ideal is
\[
    J = \ideal{y_{1}+y_{2}-2x_{1}-2x_{2},\ y_{1}y_{2}+x_{1}^{2}-2x_{1}x_{2}+x_{2}^{2}},
\]
which is prime in \(\QQ[\xbf,\ybf]\). We therefore denote $Q=J$.
        
 The integral closure of \(\QQ[\xbf,\ybf]/Q\) is simply \(\QQ[\xbf,\ybf]/Q'\) with \(Q'=J\) and the unique maximal ideal $\nfrak$ generated by $\xbf,\ybf$. The Jacobian matrix of the generators  of \(Q'\) with respect to \(\ybf\) is
        \[\begin{pmatrix}
                1&y_2\\
                1&y_1
            \end{pmatrix}\equiv\begin{pmatrix}
                1&0\\
                1&0
            \end{pmatrix}\ (\mod\nfrak),\]
      which has rank $1$ at $\nfrak$. Consequently, $\mu_A$ does {not} split over $\widehat{\Ocal_{X,o}}$.

        However, as pointed out in \cite{MR4890428}, performing the blow-up $x_1 \mapsto x_1x_2$ yields a matrix whose minimal polynomial splits completely.
        Let
        \begin{equation*}
            \tilde{A}=\begin{pmatrix}
                2x_1x_2 & x_1x_2+x_2 \\
                x_1x_2+x_2 & 2x_2
            \end{pmatrix},
        \end{equation*}
        Its minimal polynomial is
        \[\mu_{\tilde{A} }(t)=t^2+(-2x_1x_2-2x_2)t-x_1^2x_2^2+2x_1x_2^2-x_2^2,\]
        and the associated Vieta's ideal is
        \[J=\ideal{-2x_1x_2+y_1+y_2-2x_2,\ x_1^2x_2^2-2x_1x_2^2+y_1y_2+x_2^2}.\]
        The ideal \(J\) has an associated prime 
        \[Q=\<2x_{1}x_{2}-y_{1}-y_{2}+2x_{2},y_{1}^{2}+6y_{1}y_{2}+y_{2}^{2}-8y_{1}x_{2}-8y_{2}x_{2}+16x_{2}^{2}\>.\]
        The integral closure of \(\QQ[\xbf,\ybf]/Q\) is
        \[\QQ[\xbf,\ybf,w_{0,1}]/Q'.\]
        where \(Q'=\<2x_{1}x_{2}-y_{1}-y_{2}+2x_{2},\ y_{1}^{2}+6y_{1}y_{2}+y_{2}^{2}-8y_{1}x_{2}-8y_{2}x_{2}+16x_{2}^{2},\ w_{0,1}x_{2}+y_{2},\ 2w_{0,1}y_{2}+y_{1}x_{1}+5y_{2}x_{1}-3y_{1}+y_{2}+8x_{2},\ 2w_{0,1}y_{1}-y_{1}x_{1}-y_{2}x_{1}+3y_{1}+3y_{2}-8x_{2},\ w_{0,1}^{2}+2w_{0,1}x_{1}+2w_{0,1}-x_{1}^{2}+2x_{1}-1\>\).
        
        The maximal ideal \(\mathfrak{n}\) is  
\[
    \bigl\langle x_2,\; x_1,\; y_2,\; y_1,\; w_{0,1}^2+2w_{0,1}-1\bigr\rangle,
\]
and the Jacobian matrix of the generators of \(Q'\) with respect to \(\ybf\) and \(w_{0,1}\) at \(\mathfrak{n}\) is  
\[
    \begin{pmatrix}
        0&0&0&0&0&2\,w_{0,1}+2\\
        -1&0&0&-3&2\,w_{0,1}+3&0\\
        -1&0&1&2\,w_{0,1}+1&3&0
            \end{pmatrix},
\]
        This matrix has full rank in \(\QQ[\xbf,\ybf,w_{0,1}]/\nfrak\). Therefore, \(\mu_{\tilde{A}}\) completely splits in \(\widehat{\Ocal_{X,o}}\).
    \end{example}

\section{When are Spectral Projections in \(\Mat_n(\Cbf)\)?}\label{sec:diag}

Algorithm~\ref{alg:whether has proper divisor} decides whether a normal matrix $A \in \operatorname{Mat}_n(\mathcal{O}_{X,o})$ admits a unitary diagonalization; if so, it computes the solution.

%associated spectral projections. Furthermore, in the specific cases where $A$ is Hermitian or skew-Hermitian, the algorithm produces the full unitary diagonalization of $A$.

    \begin{algorithm}[htbp]
        \caption{Calculating the Unitary Diagonalization of a Normal Matrix} \label{alg:whether has proper divisor}
        \KwIn{A normal matrix \(A\in\operatorname{Mat}_n(\mathcal{O}_{X,o})\).} 
		\KwOut{If \(A\) is diagonalizable over \(\widehat{\mathcal{O}_{X,o}}\), return \textbf{Diagonalizable} together with a unitary matrix \(U\) and a diagonal matrix \(D\); otherwise \textbf{NOT Diagonalizable}.}
		
        \begin{enumerate}[label=\arabic*.]\upshape
            \item Compute the minimal polynomial \(\mu_A(t)\). Apply Algorithm \ref{alg:factor} to 
            \(\mu_A\). 
            \begin{itemize}
             \item    If the result is \textbf{False}, return \textbf{NOT Diagonalizable} and terminate.
              %  \item 
            % If the result if \textbf{True}, proceed.
             \end{itemize}
            \item Calculate the spectral projections in $\operatorname{Mat}_n(\Frac(\mathbb{C}[\xbf,\ybf,\wbf]/Q'))$:
            \[
                \Pi_k=P_k(A)\coloneq  \prod_{i\neq k}\frac{A-y_iI_n}{y_k-y_i}~(\mathrm{mod}~Q').
            \]
            \item For each triple \((k,i,j)\), write the entry \(p_{ij}^{(k)}\) of \(\Pi_k\) as a rational function
              \[
              p_{ij}^{(k)} = \frac{F_{ij}^{(k)}(\xbf,\ybf,\wbf)}{G_{ij}^{(k)}(\xbf,\ybf,\wbf)}(\mathrm{mod}~Q').
              \]
            Compute the generators of the quotient ideal \(\left(\<G_{ij}^{(k)},Q'\>:\<F_{ij}^{(k)},Q'\>\right)\).
            \begin{itemize}
            \item   
            If it contains at least one generator \(g_{ij}^{(k)}\not\equiv0(\mod\nfrak)\), 
            we compute a representative
            \[
            p_{ij}^{(k)}=\frac{F_{ij}^{(k)}(\xbf,\ybf,\wbf)}{G_{ij}^{(k)}(\xbf,\ybf,\wbf)}\equiv\frac{f_{ij}^{(k)}(\xbf,\ybf,\wbf)}{g_{ij}^{(k)}(\xbf,\ybf,\wbf)}(\mathrm{mod}~Q'),
            \]
            where \(g_{ij}^{(k)}\) is invertible on \(\widehat{\Ocal_{X,o}}\).
            \item If for any triple \((k,i,j)\), the generators of \(g_{ij}^{(k)}\) are reduced to $0$ modulo $\nfrak$, return \textbf{NOT Diagonalizable}. Otherwise continue the next step.
            \end{itemize}
        
            \item For each $k$, compute \(\Pi_k(\mod\nfrak)\in\Mat_n(\kappa(\nfrak))\).
                     
 \begin{itemize}
            \item Choose a maximal \(\kappa(\nfrak)\)-linearly independent family from the columns of \(\Pi_k (\mod{\nfrak})\) for \(k=1,\dots,d\).
            \item Lift these to the corresponding columns in \(\Pi_k\) to obtain a \(\Frac(\mathbb{C}[\xbf,\ybf,\wbf]/Q')\)-basis. 
            \item Apply the Gram–Schmidt orthogonalization process to this basis to produce a unitary matrix \(U\). (see Remark \ref{rmk:Gram-Schmidt} for a detailed explanation)  
        \end{itemize}
 Return \textbf{Diagonalizable}, together with \(U\) and \(D\coloneq U^* A U\).
       
        \end{enumerate}
    \end{algorithm}

    \begin{remark}
        Step~3 of Algorithm~\ref{alg:whether has proper divisor} fulfills two key objectives:
\begin{enumerate}
    \item \textbf{Membership verification:} To determine whether each entry \(p_{ij}^{(k)}\) of the spectral projection belongs to the ring \(\widehat{\Ocal_{X,o}}\). % This is verified via an ideal-quotient condition; failure for any entry implies the matrix is not diagonalizable over \(\widehat{\Ocal_{X,o}}\).
    
    \item \textbf{Quotient representation:} To compute, for valid entries, a representation \(p_{ij}^{(k)}=f_{ij}^{(k)}/g_{ij}^{(k)}\) where \(g_{ij}^{(k)}\) is invertible in \(\widehat{\Ocal_{X,o}}\). This explicit form is required for the residue-class computations in Step~4, which identify a maximal linearly independent family of columns.
\end{enumerate}

Notice that $p_{ij}^{(k)}$ actually belongs to $\Frac(\overline{\Ocal_{X,o}}^L)\subseteq\Frac(\Ocal_{X,o}^h)\subseteq\Frac(\widehat{\Ocal_{X,o}})$. According to Property \ref{lem:PropertesOfHenselization} \ref{lem:PropertesOfHenselization6}, if $p_{ij}^{(k)}$ lies in $\widehat{\Ocal_{X,o}}$, then $p_{ij}^{(k)}$ will belong to $(\overline{\Ocal_{X,o}}^L)_{\nfrak}$ for the maximal ideal $\nfrak=\overline{\Ocal_{X,o}}^L\cap\widehat{\mfrak_{X,o}}$ of $\overline{\Ocal_{X,o}}^L$.  Since the maximal ideals of $\overline{\Ocal_{X,o}}^L$ are Galois conjugates to each other, the maximal ideal $\nfrak$ can be arbitrary. 
\end{remark}

\begin{remark}\label{rmk:Gram-Schmidt}

Note that in Step 4 of Algorithm \ref{alg:whether has proper divisor}, one must determine how the involution $*$ on $\Cbf=\CC[\![\xbf]\!]$ acts on the eigenvalues $(\ybf\mod Q')$ of $A$. 
\begin{itemize}
    \item If $A$ is Hermitian, the eigenvalues are invariant under the involution, i.e., 
    \[ (\ybf\mod Q')^*=(\ybf\mod Q'). \]
    \item If $A$ is skew-Hermitian, the eigenvalues are skew-invariant, i.e., 
    \[ (\ybf\mod Q')^*=-(\ybf\mod Q'). \]
\end{itemize}
If $A$ is neither Hermitian nor skew-Hermitian, we utilize the decomposition $A=A_1+iA_2$, where $i=\sqrt{-1}$ and
\[ A_1=\frac{A+A^*}{2}, \quad A_2=\frac{A-A^*}{2i}. \]
Since $A$ is normal, the Hermitian matrices $A_1$ and $A_2$ commute. We can first apply Algorithm \ref{alg:whether has proper divisor} to unitarily diagonalize $A_1$. Because $A_1$ and $A_2$ commute, $A_2$ leaves the eigenspaces of $A_1$ invariant. We can therefore restrict $A_2$ to each characteristic subspace of $A_1$ (where it acts as a Hermitian operator) and recursively apply Algorithm \ref{alg:whether has proper divisor}. This process yields the simultaneous unitary diagonalization of $A_1$ and $A_2$, and consequently of $A$.

    %Then we compute the characteristic polynomials $\chi_{A_1}(t)$ and $\chi_{A_2}(t)$ of $A_1$ and $A_2$, respectively. Next, introduce the variables $\ubf=(u_1,\dots,u_n)$ and $\vbf=(v_1,\dots,v_n)$ to represent the roots of $\chi_{A_1}(t)$ and $\chi_{A_2}(t)$, respectively. Then $\ubf+i\vbf$ are exactly the roots of $\chi_A$, i.e., the eigenvalues of $A$. The polynomial $\chi_{A_1},\chi_{A_2}$ and $\chi_A$ give us three Vieta's ideals $J_1(\ubf),J_2(\vbf)$ and $J(\ubf+i\vbf)$ in variables $\ubf,\vbf$ and $\ubf+i\vbf$. Denote $\widetilde{J}=J_1(\ubf)+J_2(\vbf)+J(\ubf+i\vbf)\subset\Ocal_{X,o}[\ubf,\vbf]$. Then we need to compute an arbitrary minimal prime $Q$ of the universal decomposition algebra $\Ocal_{X,o}[\ubf,\vbf]/\widetilde{J}$ and calculate the integral closure $\Ocal_{X,o}[\ubf,\vbf,\wbf]/Q'$ of $\Ocal_{X,o}[\ubf,\vbf]/Q$ in its fraction field. Now 
\end{remark}

We present the following example to demonstrate Algorithm \ref{alg:whether has proper divisor}.

    \begin{example}[Continuing Example \ref{exp:rellich}]
        For the Hermitian matrix
        \[\tilde{A}=\begin{pmatrix}
                2x_1x_2 & x_1x_2+x_2 \\
                x_1x_2+x_2 & 2x_2
            \end{pmatrix},\]
        we have known that its minimal polynomial splits in \(\Cbf=\widehat{\Ocal_{X,o}}\). In fact, since $\tilde{A}$ is Hermitian, its minimal polynomial splits in the center of involution $\Rbf=(\widehat{\Ocal_{X,o}})^*$, implying that the eigenvalues of $\tilde{A}$ are all real power series. By Lagrange interpolation, we obtain
        \[
        \Pi_1=\begin{pmatrix}
                \frac{w_{0,1}x_1-w_{0,1}+3x_1^2+1}{4(x_1^2+1)} & \frac{w_{0,1}x_1+w_{0,1}+x_1^2+2x_1+1}{4(x_1^2+1)} \\
                \frac{w_{0,1}x_1+w_{0,1}+x_1^2+2x_1+1}{4(x_1^2+1)}  & \frac{-w_{0,1}x_1+w_{0,1}+x_1^2+3}{4(x_1^2+1)} 
            \end{pmatrix},\]
        \[\Pi_2=\begin{pmatrix}
                \frac{-w_{0,1}x_1+w_{0,1}+x_1^2+3}{4(x_1^2+1)} & \frac{-w_{0,1}x_1-w_{0,1}-x_1^2-2x_1-1}{4(x_1^2+1)} \\
                \frac{-w_{0,1}x_1-w_{0,1}-x_1^2-2x_1-1}{4(x_1^2+1)}  & \frac{w_{0,1}x_1-w_{0,1}+3x_1^2+1}{4(x_1^2+1)}
            \end{pmatrix}.
        \]
 Choosing the first column, respectively, of \(\Pi_1,\Pi_2\) and performing the Gram-\-Schmidt process, where \(w_{0,1}^*=w_{0,1}\) is real in \(\widehat{\Ocal_{X,o}}\), we obtain
        \[
        U=\begin{pmatrix}
                \frac{w_{0,1}x_1-w_{0,1}+3x_1^2+1}{\sqrt{\alpha}} & \frac{-w_{0,1}x_1+w_{0,1}+x_1^2+3}{\sqrt{\beta}} \\
                \frac{w_{0,1}x_1+w_{0,1}+x_1^2+2x_1+1}{\sqrt{\alpha}}  & \frac{-w_{0,1}x_1-w_{0,1}-x_1^2-2x_1-1}{\sqrt{\beta}}
            \end{pmatrix}
        \]
        where \(\alpha=4w_{0,1}x_{1}^{3}-4w_{0,1}x_{1}^{2}+4w_{0,1}x_{1}-4w_{0,1}+12x_{1}^{4}+16x_{1}^{2}+4,\ \beta=-4w_{0,1}x_{1}^{3}+4w_{0,1}x_{1}^{2}-4w_{0,1}x_{1}+4w_{0,1}+4x_{1}^{4}+16x_{1}^{2}+12\), and
        \[D=\begin{pmatrix}
                x_2\left(x_1+1+\sqrt{2x_1^2+2}\right) & 0 \\
                0  & x_2\left(x_1+1-\sqrt{2x_1^2+2}\right)
            \end{pmatrix}.\]
    \end{example}
    
    \section{Conclusions and Future Work}\hfill

We have established a rigorous algebraic framework and a constructive algorithm for the unitary diagonalization of normal matrices over formal power series rings.  %By generalizing the classical spectral theorem, our approach enables effective symbolic computation via the splitting algebra.
A primary bottleneck of the current algorithm is the calculation of the integral closure. Future work will focus on circumventing this expensive step, potentially by exploiting Galois symmetries or developing specialized lifting algorithms.

Additionally, since many applications prioritize the asymptotic behavior of diagonalized matrices, we plan to explore hybrid approaches. Combining our algebraic results with analytic representations -- such as defining power series using recurrence relations~\cite{van2019effective} -- offers a promising direction for efficient computation and analysis.

    \section*{Acknowledgments}
        We thank Prof. Li, Zijia for his help with the algorithm experiments. We thank Jiaqi Wang for his help in finding references. We would like to thank Joris van der Hoeven for {(sharing interesting fast arithmetics over the power series ring)}, and Èric schost for sharing with us an article on polynomial factorisation. Hao Liang would like to thank the instructions and discussions of Prof. Hu, Yongquan and Prof. Li, Shizhang on  ramification theory.  Zihao Dai, Hao Liang and  Lihong  Zhi are supported by the National Key R$\&$D Program of China 2023YFA1009401 and the Strategic Priority Research Program of Chinese Academy of Sciences under Grant XDA0480501.

\bibliographystyle{alpha}
\bibliography{Biblio,Article,Book}

\end{document}